\documentclass[11pt]{article}

\usepackage{a4wide}
\usepackage{latexsym}
\usepackage{amsfonts}
\usepackage{amsthm}
\usepackage{amsmath}
\usepackage{amsopn,amscd}
\usepackage{hyperref}

\def\bra{[\,\cdot\, , \,\cdot\,]}
\def\braaa{[\,\cdot\, ,\,\cdot\,, \, \cdot\,]}
\def\bran{[\,\cdot\, ,\ldots, \, \cdot\,]}
\def\bracn{\{\,\cdot\, ,\ldots, \, \cdot\, |  \, \cdot\, \}}

\def\act{(\,\cdot\, , \,\cdot\,)}
\def\actt{\langle\cdot,\cdot\,;\cdot\rangle}
\def\Sigm{\mathrm S}

\def\AA{\mathcal A}
\def\QQ{\mathcal Q}

\def\g{\mathfrak g}
\def\Sigmc{\mathrm S^{\mathrm c}}
\def\Tc{\mathrm T^{\mathrm c}}
\def\mani{M}
\def\Aalt{\mathrm{Alt}}

\def\dho{\mathcal D}
\def\cpartial{\partial_{\bra}}
\def\apartial{\partial^{\bra}}
\def\multialgebra{multi derivation chain algebra}
\def\mmd{multi derivation Maurer-Cartan}
\def\ppt{filtered degree $-1$ filtration lowering coderivation\/}
\def\fdm{filtered degree $-1$ morphism\/}

\newtheorem*{t1.9}{Theorem 1.9}
\newtheorem*{Theorem 4.13}{Theorem 4.13}

\newtheorem{thm}{Theorem}[section]
\newtheorem{cor}[thm]{Corollary}

\newtheorem{prop}[thm]{Proposition}

\theoremstyle{definition}
\newtheorem{defi}[thm]{Definition}

\theoremstyle{remark}
\newtheorem{rema}[thm]{Remark}

\numberwithin{equation}{section}
\title
{Multi derivation Maurer-Cartan algebras and sh-Lie-Rinehart algebras}
\author{
J.~Huebschmann
\\[0.3cm]
 USTL, UFR de Math\'ematiques\\
CNRS-UMR 8524
\\
Labex CEMPI (ANR-11-LABX-0007-01)
\\
59655 Villeneuve d'Ascq Cedex, France\\
Johannes.Huebschmann@math.univ-lille1.fr
 }

\long
\def\MSC#1\EndMSC{\def\arg{#1}\ifx\arg\empty\relax\else
      {\par\narrower\noindent
      2010 Mathematics Subject Classification: #1\par}\fi}
\long
\def\KEY#1\EndKEY{\def\arg{#1}\ifx\arg\empty\relax\else
    {\par\narrower\noindent
      Keywords and Phrases: #1\par}\fi}

\begin{document}
\maketitle
\begin{abstract} We extend the classical characterization of a 
finite-dimensional Lie algebra $\g$ in terms of its
Maurer-Cartan algebra---the familiar differential graded algebra
of alternating forms on $\g$ with values in the ground field, endowed
with the standard Lie algebra cohomology operator---to sh Lie-Rinehart 
algebras. To this end, we first develop a characterization of
sh Lie-Rinehart algebras in terms of 
differential graded cocommutative coalgebras and
Lie algebra twisting cochains that extends the nowadays standard
characterization of an ordinary sh Lie algebra
(equivalently: Linfty algebra)
in terms of its associated generalized Cartan-Chevalley-Eilenberg coalgebra.
Our approach avoids any higher brackets but reproduces these brackets
in a conceptual manner. The new technical tool we develop is
a notion of filtered \multialgebra, somewhat more general
than the standard notion of a multicomplex endowed with a compatible
algebra structure.
The crucial observation, just as for ordinary Lie-Rinehart algebras,
is this: 
For a general sh Lie-Rinehart algebra,
the generalized Cartan-Chevalley-Eilenberg 
operator on the corresponding graded algebra
involves two operators, one coming from the sh Lie algebra structure
and the other from the generalized action on the corresponding algebra;
the sum of the operators is defined on the algebra while the operators
are individually defined only on a larger ambient algebra.
We illustrate the structure with quasi Lie-Rinehart algebras.
\end{abstract}

\MSC 16E45, 16T15, 17B55, 17B56, 17B65, 17B70, 18G10, 55P62
\EndMSC

\KEY Maurer-Cartan algebra, sh Lie-Rinehart algebra, $L_{\infty}$ algebra,
higher homotopies, quasi Lie-Rinehart algebra, foliation
\EndKEY

\tableofcontents

\section{Introduction}
A finite-dimensional Lie algebra $\g$ can be characterized in terms of its
{\em Maurer-Cartan\/} algebra, that is, the algebra of alternating forms
on $\g$ with the C(artan-)C(hevalley-)E(ilenberg) differential.
The same is true of a Lie-Rinehart algebra
$(A,L)$ when $L$ is finitely generated and projective as an $A$-module.
A Lie-Rinehart algebra $(A,L)$ is a pair that consists of
a commutative algebra $A$ and a Lie algebra $L$ together with
an $A$-module structure on $L$ and an $L$-action on $A$
by derivations such that two obvious axioms are satisfied;
these axioms are
modeled on the standard example
$(A,L)=(C^{\infty}(\mani),\mathrm{Vect}(\mani))$
that consists of the smooth functions 
$C^{\infty}(\mani)$ and smooth vector fields $\mathrm{Vect}(\mani)$
on a smooth manifold $\mani$.
Given a Lie-Rinehart algebra $(A,L)$,
the CCE operator on the CCE
algebra $\Aalt(L,A)$ involves two derivations
$\apartial$ and $\partial^t$,
the first coming from the Lie bracket and the second
from the $L$-action on $A$, and the sum
$d=\apartial+\partial^t$, at first defined
on $\Aalt(L,A)$, passes to a derivation on
the $A$-valued $A$-multilinear forms $\Aalt_A(L,A)$
and turns this algebra into differential graded $R$-algebra,
even though the individual derivations $\apartial$ and $\partial^t$
do not necessarily descend,
and we refer to the resulting
differential graded $R$-algebra
$(\Aalt_A(L,A),d)$
as the {\em Maurer-Cartan algebra\/}
associated to the data. 
When $L$ is finitely generated and projective as an $A$-module,
this Maurer-Cartan algebra characterizes
the Lie-Rinehart algebra.
In the situation of
the standard example
$(A,L)=(C^{\infty}(\mani),\mathrm{Vect}(\mani))$,
the Maurer-Cartan algebra is the de Rham algebra
of the underlying smooth manifold.

In this paper we generalize the Maurer-Cartan characterization
to sh Lie-Rinehart algebras.
The idea of an sh Lie algebra or, equivalently, $L_{\infty}$ algebra,
has a history \cite{MR2762538}, \cite{MR2640649};
we only mention that the $A_{\infty}$-algebra concept,
prior to the $L_{\infty}$-algebra concept,
was introduced by J. Stasheff in the 1960-s, cf. 
\cite{MR2762538}, \cite{MR2640649}.
Sh Lie-Rinehart algebras were introduced in
\cite{MR1854642} (part of a thesis supervised by J. Stasheff).
In \cite{MR2103009} we introduced quasi Lie-Rinehart algebras
as a higher homotopies generalization of ordinary Lie-Rinehart algebras.
Quasi Lie-Rinehart algebras actually arise in mathematical nature
in the theory of foliations \cite{MR2103009}.
In this paper we develop a Maurer-Cartan type characterization
of sh Lie-Rinehart algebras.
This recovers quasi Lie-Rinehart algebras
as a special case of sh Lie-Rinehart algebras.
The new technical tool we introduce for that purpose
is a notion of \multialgebra, more flexible than the traditional 
concept of a multicomplex endowed with a compatible algebra structure
(as we hope to demonstrate in this paper) and also somewhat more general,
cf. Remarks \ref{classical1} and \ref{classical} below.
In Theorems \ref{char2} and \ref{olr2}
below, we show how sh Lie-Rinehart algebras can be characterized
in terms of the newly developed notion of \multialgebra.
Suffice it to mention here that a 
\multialgebra\  is a graded commutative algebra
$\AA$ together with a filtration $\AA^0\supseteq \AA^1 \ldots$
and a family of degree $-1$ derivations $\{\dho_j\}_{j \geq 0}$
such that $\dho_j$ lowers filtration by $j$ and
such that $\sum_{j \geq 0} \dho_j$ is a differential.
See Section \ref{maintechnical} for details.
The crucial observation now is this:
An sh Lie-Rinehart algebra $(A,L)$ leads to
a 
\multialgebra\  
$\left(\AA,\dho_0,\dho_1, \dho_2,\ldots\right)$
of graded symmetric $A$-multilinear forms but,
beware, the differential  $\sum_{j \geq 0} \dho_j$ is only linear over
the ground ring,
such that, for $j \geq 1$, each operator
$\dho_j$ has the form $\dho_j=\apartial_j+\partial^{t_j}$ and such that,
just as in the case of ordinary Lie-Rinehart algebras,
while $\apartial_j$ and $\partial^{t_j}$ are not individually
defined on $\AA$
(only on a larger ambient graded algebra),
their sum $\dho_j$ is defined on $\AA$.
Under a suitable additional assumption ($A$-reflexivity of $L$),
these \multialgebra\  structures then characterize sh Lie-Rinehart 
algebra structures on $(A,L)$.
A salient feature is that an sh Lie structure
of $L$ lives on the cofree differential graded cocommutative coalgebra
$\Sigmc[sL]$ on $sL$ over the ground ring
whereas the Maurer-Cartan algebra
characterization of an sh Lie-Rinehart structure
is phrased in terms of an algebra
of graded symmetric $A$-multilinear forms on $sL$, viewed
as a  graded $A$-module.

To make the results more easily accessible, in Section \ref{OLR0}
we explain first the special
case of ordinary Lie-Rinehart algebras in a language tailored 
to the general situation. 
In Section 3 we spell out some technical tools
that are indispensable thereafter. 
Here we borrow from the theory of homological perturbations,
cf. e.g., \cite{MR2640649} and \cite{MR2762538}.
We spell out 
the main results
related with sh Lie-Rinehart algebras in Section 5.
Theorem \ref{shlrchar0} says that,
given the relevant data,
under the hypothesis spelled out there,
these data constitute an sh Lie-Rinehart algebra if and only if
they induce a \multialgebra\  structure
on the corresponding object.
Theorem \ref{shlrchar} says that, 
under the stronger hypothesis of this theorem,
every \multialgebra\  structure on
of the kind under discussion arises from a unique
sh Lie-Rinehart algebra structure. See also Remark \ref{remprecise}
below.

On the technical side we note here that we avoid
any \lq\lq bracket yoga\rq\rq. In $L_{\infty}$-technology, it is nowadays
common to use a bracket zoo which necessarily comes with
complications related with signs etc. Our approach in terms of
differential graded cocommutative coalgebras and Lie algebra twisting 
cochains avoids spelling out explicitly 
any of the corresponding brackets and
takes care of any of the complications by itself, once the
Eilenberg-Koszul sign convention has been implemented.

The terminology \lq Maurer-Cartan algebra\rq\  goes back at least to \cite{MR1425752};
among many other things, van Est noticed that the idea 
of a Maurer-Cartan algebra
was used by E. Cartan already in 1936
to characterize the structure of Lie groups and Lie algebras.

A construction aimed at characterizing 
sh Lie-Rinehart algebras in terms of  Maurer-Cartan algebras
has been developed in \cite{vitagone}. The approach in that paper
tames the corresponding bracket zoo. 
In
\cite{MR2110368} and \cite{MR2103009} we used the terminology
\lq\lq multialgebra\rq\rq\ for what we now refer to as
a \multialgebra.
We hope this avoids confusion with a well
 established notion of multialgebra
in the literature that has a meaning very different from
that of \multialgebra, cf. e.g., \cite{MR0146103}.

We are much indebted to J. Stasheff, for many discussions on the topic,
for having enthusiastically insisted 
that the relationship between the various notions discussed in this paper
be conclusively clarified, and for a number of valuable comments on 
a draft of the manuscript.

\section{Ordinary Lie-Rinehart algebras}
\label{OLR0}

For ease of exposition, we explain first the case of ordinary 
Lie-Rinehart algebras, in language and notation tailored to our purposes.
We hope this will provide a road map for the reader so that he can more 
easily digest the material in later sections.

Under suitable circumstances, a Lie-Rinehart algebra
can be characterized in terms of its {\em Maurer-Cartan algebra\/}.
We will give a precise statement as Theorem \ref{olr} below.
To prepare for it,
let $R$ be a commutative ring with 1;
henceforth the unadorned tensor product
refers to the tensor product over $R$.
Let $A$ be a commutative $R$-algebra; then the commutator bracket turns
the $A$-module $\mathrm{Der}(A|R)$ of derivations of $A$
into an $R$-Lie algebra (beware: this is no longer true when
$A$ is non-commutative). Further, let
$L$ be an $A$-module,
$\bra\colon L \times L \to L$ a skew-symmetric pairing, and
\begin{equation}
\vartheta\colon L \longrightarrow \mathrm{Der}(A|R)
\label{derA}
\end{equation}
an $R$-linear map.
Given $\alpha\in L$ and $a \in A$ we write
$\alpha(a)=(\vartheta(\alpha))(a)$.
The pair $(A,L)$ is said to constitute a Lie-Rinehart algebra
when the pieces of structure satisfy  two obvious axioms
modeled
on the pair  $(A,L)= (A,\mathrm{Der}(A|R))$, cf. \cite{poiscoho}.
These axioms read
\begin{align} (a\,\alpha)(b) &= a\,(\alpha(b)), \quad \alpha \in L,
\,a,b \in A, \label{1.1.a}
\\
\lbrack \alpha, a\,\beta \rbrack &= a\, \lbrack \alpha, \beta
\rbrack + \alpha(a)\,\beta,\quad \alpha,\beta \in L, \,a \in A.
\label{1.1.b}
\end{align}

Given only the pieces of structure
$A,L,\bra,\eqref{derA}$,
consider the $R$-algebra $\Aalt(L,A)$
of $A$-valued $R$-multilinear alternating forms on $L$, and
define two $R$-linear derivations $\partial^{t}$ and $\apartial$
on $\Aalt(L,A)$ by
the familiar expressions
\begin{align}
(-1)^{n-1}(\partial^{t}f)&(\alpha_1,\dots,\alpha_n)
=
\sum_{i=1}^n (-1)^{(i-1)}
\alpha_i(f (\alpha_1, \dots\widehat{\alpha_i}\dots, \alpha_n))
\label{dt}
\\
(-1)^{n-1}(\apartial f)&(\alpha_1,\dots,\alpha_n)
= 
\sum_{1 \leq j<k \leq n} (-1)^{(j+k)}f(\lbrack \alpha_j,\alpha_k \rbrack,
\alpha_1, \dots\widehat{\alpha_j}\dots\widehat{\alpha_k}\dots,\alpha_n).
\label{dbra}
\end{align}
The sign $(-1)^{n-1}=(-1)^{|f|}$ is, 
perhaps, not entirely classical.
According to the Eilenberg-Koszul convention,
it is the correct sign, cf. \eqref{homd} and
\eqref{aexplicit} below.
We will justify the notation $\partial^{t}$ shortly;
suffice it to note for the moment that,
when $\bra$ is a Lie bracket and \eqref{derA}
an $L$-action on $A$ by derivations, the operator 
 $\partial^{t}$ 
arises from a Lie algebra twisting cochain $t$
that recovers the $L$-action on $A$. 
Let \begin{equation}
d=\partial^{t} +\apartial .
\end{equation}
The reader will notice right away the following. 

\begin{prop}
\label{rightaway}
When
$\bra$ is a Lie bracket and \eqref{derA} a morphism of $R$-Lie algebras,
the derivation $d$ is 
a differential, in fact, 
the classical C(artan-)C(hevalley-)E(ilenberg)
operator,
and
the differential graded algebra $(\Aalt(L,A),d)$ computes the 
CCE
cohomology of $L$ with coefficients in $A$.
\end{prop}

Using terminology that goes back at least to \cite{MR1425752},
we refer to a differential graded algebra of the kind 
$(\Aalt(L,A),d)$ as a {\em Maurer-Cartan\/}
algebra. 

Concerning  Lie-Rinehart algebras,
a crucial observation is now
the following.

\begin{prop}
\label{cruc}
When
$(A,L)$ is a Lie-Rinehart algebra,
the derivation $d$ descends to an $R$-linear differential
on $\Aalt_A(L,A)$,
even though this is not true of the individual operators
$\partial^{t}$ and $\apartial$ unless $A=R$ and $\partial^{t}$ is trivial.
\end{prop}

This observation has a long history; see e.g., \cite{MR2075590}
for a survey. Under such circumstances, we  refer
to the differential graded $R$-algebra 
$(\Aalt_A(L,A),d)$ as the {\em Maurer-Cartan
algebra\/} 
associated to the Lie-Rinehart algebra
and to $d$ as its {\em CCE operator\/}.

We spell out the proof since it clearly shows
how the Lie-Rinehart axioms imply that
the derivation $d$ descends to an $R$-linear differential
on $\Aalt_A(L,A)$,
even though this is not necessarily true of the individual operators
$\partial^{t}$ and $\apartial$.

\begin{proof}[Proof of Proposition {\rm \ref{cruc}}]
We explain the crucial observation; this will help the reader
understand the general case in Theorem \ref{multishp} below:
Let $u \in A$ and $\varphi \in \Aalt^1_A(L,A)=\mathrm{Hom}_A(L,A)$.
Our aim is to show that $du$ is $A$-linear
and that $d\varphi$ is $A$-bilinear. Given $\alpha,\beta \in L$ and $a \in A$,
exploiting the $A$-linearity of $\varphi$ and the two Lie-Rinehart axioms
\eqref{1.1.a} and \eqref{1.1.b}, we find
\begin{align*}
du(a\alpha)&=\partial^{t}u(a\alpha)
\\
&=(a\alpha)(u)=a((\alpha)(u))=adu 
\\
d\varphi(\alpha,a\beta)
&=
\partial^{t}\varphi(\alpha,a\beta) +\apartial\varphi(\alpha,a\beta)
\\
\partial^{t}\varphi(\alpha,a\beta)&=a\partial^{t}\varphi(\alpha,\beta)
+\alpha(a)\varphi(\beta)
\\
\apartial\varphi(\alpha,a\beta)&=a\apartial\varphi(\alpha,\beta)
-\alpha(a)\varphi(\beta)
\\
d\varphi(\alpha,a\beta)
&=
a\left(\partial^{t}\varphi(\alpha,\beta) +\apartial\varphi(\alpha,\beta)\right)
\\
&=
ad\varphi(\alpha,\beta).
\end{align*}
A similar reasoning establishes the $A$-multilinearity
in an arbitrary degree.
\end{proof}

\begin{rema} \label{confirm}
There is no need to confirm 
the $A$-multilinearity
in (upper) degrees $>1$: When $L$ is finitely generated
as an $A$-module,
as an $R$-algebra,
the graded $A$-algebra $\Aalt_A(L,A)$
is generated by its elements in degree
$0$ and (upper) degree $1$.
When $L$ is not finitely generated
as an $A$-module, in a suitable topology,
the graded $A$-subalgebra of $\Aalt_A(L,A)$
generated by its elements in degree
$0$ and (upper) degree $1$
is dense in  $\Aalt_A(L,A)$.
\end{rema}

We will now develop an alternate 
characterization of  Lie-Rinehart algebras,
to be given as Theorem \ref{immediate4} below.
This characterization
will pave the way for
developing a notion of sh Lie-Rinehart algebra in
Section \ref{sh} below.
To this end,
let $sL$ be the suspension of $L$, that is, $sL$ is
the $A$-module $L$ regraded up by one,
and consider the (graded) exterior $R$-algebra $\Lambda[sL]$;
to avoid misunderstanding or confusion, we note that we take 
$\Lambda[sL]$ to be the graded symmetric $R$-algebra on $sL$.
The familiar shuffle diagonal turns $\Lambda[sL]$
into an $R$-bialgebra, in particular, into an $R$-coalgebra,
and the skew-symmetric bracket $\bra$ on $L$
(not assumed to satisfy the Jacobi identity)
determines and is determined by a degree $-1$ coderivation
$\cpartial$ on  $\Lambda[sL]$.
This coderivation induces the derivation $\apartial$
on  $\Aalt(L,A)\cong \mathrm{Hom}(\Lambda[sL],A)$
given by \eqref{dbra};
at this stage  the sign in \eqref{dbra} is the correct one.
Moreover, 
$\Aalt(L,\mathrm{Der}(A|R))\cong
\mathrm{Hom}(\Lambda[sL],\mathrm{Der}(A|R))$ 
acquires a graded Lie algebra structure
as well as a graded $\Aalt(L,A)$-module structure,
the action of $\mathrm{Der}(A|R)$ on $A$ extends to an action
\begin{equation}
\bra\colon \Aalt(L,\mathrm{Der}(A|R))
\times
\Aalt(L,A)
\longrightarrow
\Aalt(L,A)
\label{derA2}
\end{equation}
by derivations
and, with the structure of mutual interaction,
the pair
\begin{equation}
(\Aalt(L,A),\Aalt(L,\mathrm{Der}(A|R)))\cong
\left(\mathrm{Hom}(\Lambda[sL],A),\mathrm{Hom}(\Lambda[sL],\mathrm{Der}(A|R))
\right) 
\label{mutual1}
\end{equation}
constitutes a graded Lie-Rinehart algebra.
This graded Lie-Rinehart algebra is a special case
of \eqref{mutual2} below.
The $R$-linear map $\vartheta\colon L \to \mathrm{Der}(A|R)$, cf. \eqref{derA},
determines (and is determined by) 
the degree $-1$ morphism 
\begin{equation}
t=\vartheta\circ s^{-1}\colon sL \longrightarrow \mathrm{Der}(A|R),
\label{tvartheta}
\end{equation}
the composite of the desuspension $s^{-1}$ with $\vartheta$,
and
$\partial^{t}=[t,\,\cdot\,]$ 
is the degree $-1$ derivation 
on $\mathrm{Hom}(\Lambda[sL],A)\cong \Aalt(L,A)$
given by \eqref{dt}.
The various operators are now related by the identity
\begin{align}
\left([\apartial,\partial^{t}]
+\partial^{t}\partial^{t}\right)
\varphi
&=\left[t\cpartial +t\wedge t,\varphi\right],\ \varphi \in \Aalt(L,A).
\label{exploit0}
\end{align}
Notice that the composite
\[
\begin{CD}
\Lambda[sL]@>{\cpartial}>>\Lambda[sL] @>{t}>> \mathrm{Der}(A|R)
\end{CD}
\]
is a homogeneous degree $-2$ member of
$\mathrm{Hom}(\Lambda[sL],\mathrm{Der}(A|R))$
in an obvious manner, and
here and below we use the familiar 
notation $t\wedge t=\tfrac 12[t,t]$.
Later in the paper, we will generalize the identity 
\eqref{exploit0} to \eqref{exploit1}.

Let $C$ be a coaugmented differential graded cocommutative coalgebra and
$\g$ a differential graded Lie algebra. Then the cup bracket $\bra$
induced by the diagonal of $C$ and the bracket $\bra$ of $\g$
(where the notation $\bra$ is abused)
turns $\mathrm{Hom}(C,\g)$ into a differential graded Lie algebra,
the differential $\dho$ on  $\mathrm{Hom}(C,\g)$
being the ordinary Hom-differential;
a {\em Lie algebra twisting cochain\/}
is a homogeneous morphism $t\colon C \to \g$
of $R$-modules of degree $-1$ 
whose composite with the coaugmentation $\eta\colon R \to C$
is zero and which, with the notation $t\wedge t =\tfrac 12 [t,t]$,
satisfies the identity
\begin{equation}
\dho t + t\wedge t =0.
\label{twc}
\end{equation}
See e.g., 
\cite{huebstas}, 
\cite{MR0436178},
\cite{MR0258031}.
The sign here is the same as that in \cite{MR0258031};
it differs from that in \cite{huebstas}.
The present sign convention simplifies our formulas.

The following is immediate; we spell it our for ease of exposition.

\begin{prop}\label{immediate3}
{\rm (i)}
The bracket $\bra$ on $L$ is a Lie bracket, i.e., satisfies
the Jacobi identity if and only if $\cpartial\cpartial = 0$.

\noindent
{\rm (ii)} Suppose that the bracket $\bra$ on $L$ is a Lie bracket
and let $t$ be the degree $-1$ morphism of $R$-modules given
 by {\rm \eqref{tvartheta}}.
Then the following are equivalent:
\begin{itemize}
\item The morphism $t$ is a Lie algebra twisting cochain
$(\Lambda[sL],\cpartial) \to \mathrm{Der}(A|R)$;
\item the degree $-1$ morphism $t$ satisfies the identity $t\cpartial +t\wedge t=0$;
\item the degree zero morphism $\vartheta$, cf. \eqref{derA}, 
is a morphism of $R$-Lie algebras. \qed
\end{itemize}
\end{prop}

The following observation
characterizes a Lie-Rinehart algebra structure
on $(A,L)$
entirely in terms of the  corresponding coderivation on $\Lambda[sL]$
and the corresponding degree $-1$ morphism $\Lambda[sL] \to \mathrm{Der}(A|R)$.

\begin{thm}\label{immediate4}
Given the data
$(A,L,\bra,\vartheta)$, as before,
let $t=\vartheta \circ s^{-1}\colon \Lambda[sL] \to \mathrm{Der}(A|R)$, 
cf. {\rm \eqref{tvartheta}}.
The following are equivalent.

\noindent
{\rm (i)} The data
$(A,L,\bra,\vartheta)$
constitute a Lie-Rinehart algebra.

\noindent
{\rm (ii)} The coderivation $\cpartial$ on
 $\Lambda[sL]$
is a differential, i.~e.,
$\cpartial\cpartial=0$,
the degree $-1$
morphism $t$
is $A$-linear, and  $\cpartial$  and $t$ are related by the following 
identities:
\begin{align}
t\cpartial +t\wedge t&=0
\\
\cpartial(\alpha_1,a\alpha_2)
&= (t(\alpha_1)(a))\alpha_2+a\cpartial(\alpha_1,\alpha_2),\ \alpha_1,\alpha_2\in sL,\ a\in A. 
\label{immlr}
\end{align}
\end{thm}

\begin{proof} This is straightforward and left to the reader.
\end{proof}

\begin{cor}
\label{char1}
Suppose that the $A$-module $L$ has the property that
the canonical map 
\begin{equation}
L \longrightarrow \mathrm{Hom}_A(\mathrm{Hom}_A(L,A),A)
\label{cano0}
\end{equation}
from $L$ to its double $A$-dual 
$\mathrm{Hom}_A(\mathrm{Hom}_A(L,A),A)$ is an injection
of $A$-modules.
Then the pair $(A,L)$ constitutes a Lie-Rinehart algebra if and only if
the derivation
$d=\partial^{t} +\apartial$ on
$\Aalt(L,A)$ passes to an $R$-linear differential on
$\Aalt_A(L,A)$,
necessarily a derivation.
\end{cor}

The proof is straightforward, cf. e.g., \cite{MR2103009} (Lemma 2.2.11).
In Theorem \ref{char2} below,
we will generalize the sufficiency claim.
To prepare for this generalization,
we will now give a technically  more involved proof
of Corollary \ref{char1}
which extends to the more general situation of Theorem \ref{char2},
see also Remark \ref{instructive} below.

\begin{proof}[Proof of Corollary {\rm \ref{char1}}]
Proposition \ref{cruc} shows that the condition is necessary.
To show that it is sufficient,
suppose that the derivation
$d=\partial^{t} +\apartial$ on
$\Aalt(L,A)$ passes to an $R$-linear differential on
$\Aalt_A(L,A)$.

Let $a\in A$. Since $\apartial a=0$,
\begin{align*}
0=dda&=  (\apartial +\partial^{t})(\apartial +\partial^t) a
\\
&=\left([\apartial,\partial^{t}]
+\partial^{t}\partial^{t}\right)a
\\
&=\left[t\cpartial +t\wedge t,a\right].
\end{align*}
Since $a$ is arbitrary, we conclude
\[
t\cpartial +t\wedge t =0,
\]
whence  
$\vartheta \colon L \to \mathrm{Der}(A|R)$, cf. \eqref{derA},
is compatible with the brackets.
Consequently, on $\Aalt(L,A)$ (beware: not on $\Aalt_A(L,A)$,
since this would not even make sense on $\Aalt_A(L,A)$)
\begin{equation}
[\apartial,\partial^{t}]
+\partial^{t}\partial^{t} =0 .
\label{sense}
\end{equation}

Next, let  $\varphi \in \mathrm{Hom}_A(sL,A)$,
and view $\varphi$ as a member of  $\mathrm{Hom}(sL,A)$.
Then
\begin{align*}
0=dd\varphi&=  (\apartial +\partial^{t})(\apartial +\partial^t) \varphi
\\
&=\left(\apartial\apartial +[\apartial,\partial^{t}]
+\partial^{t}\partial^{t}\right)\varphi
\\
&=\apartial\apartial \varphi +\left[t\cpartial +t\wedge t,\varphi\right]
\\
&=\apartial\apartial \varphi .
\end{align*}
Let $x \in \Lambda_3[sL]$. Then
\[
0=(dd\varphi)(x)=
(\apartial\apartial \varphi)(x)= 
\varphi(\cpartial\cpartial (x)) \in A .
\]
Since $\varphi$ is arbitrary, and since
\eqref{cano0} is injective,
we conclude that $\cpartial\cpartial (x) =0 \in sL$.
Since $x$ is arbitrary, we see that
$\cpartial\cpartial=0\colon \Lambda_3[sL]\to  \Lambda_1[sL] =sL$,
whence the bracket $\bra$ on
$L$ satisfies the Jacobi identity.

Let $a,b \in A$ and $\alpha \in sL$. 
Since $\apartial(a)=0$, the hypothesis of the corollary implies that
\begin{equation*}
b(t(\alpha))(a)= b(\partial^t a)(b\alpha)=(\partial^t a)(b\alpha)=(t(b\alpha))(a)
\end{equation*}
whence, since $a$ is arbitrary, 
$t$ is $A$-linear
or, in other words, the data satisfy
the axiom \eqref{1.1.a}.

Finally, let $a \in A$, $\alpha_1,\alpha_2 \in sL$,
and let $\varphi \in \mathrm{Hom}_A(sL,A)$.
Then
\begin{equation}
(\partial^t\varphi)(\alpha_1,a\alpha_2)
=a (\partial^t\varphi)(\alpha_1,\alpha_2) 
-\varphi(((t\alpha_1)(a))\alpha_2).
\end{equation}
Indeed,
\begin{equation*}
\begin{aligned}
(\partial^t\varphi)(\alpha_1,a\alpha_2)&=[t,\varphi](\alpha_1,a\alpha_2)
\\
&=a [t,\varphi](\alpha_1,\alpha_2)-((t\alpha_1)(a))\varphi(\alpha_2)
\\
&=a (\partial^t\varphi)(\alpha_1,\alpha_2) 
-\varphi(((t\alpha_1)(a))\alpha_2).
\end{aligned}
\end{equation*}
Moreover, 
\begin{equation}
\apartial \varphi(\alpha_1, a\alpha_2) =\varphi(\cpartial
(\alpha_1, a\alpha_2)).
\end{equation}
Since $\partial^t+\apartial$ passes to
$\Aalt_A(L,A)$, we conclude
\begin{align*}
(\partial^t+\apartial)\varphi(\alpha_1, a\alpha_2)
&=
a(\partial^t+\apartial)\varphi(\alpha_1, \alpha_2)
\\
a (\partial^t\varphi)(\alpha_1,\alpha_2) 
-\varphi(((t\alpha_1)(a))\alpha_2)
+\varphi(\cpartial
(\alpha_1, a\alpha_2))&=
a(\partial^t\varphi)(\alpha_1, \alpha_2)
+ a \varphi (\cpartial(\alpha_1, \alpha_2))
\\
\varphi(\cpartial
(\alpha_1, a\alpha_2)
-((t\alpha_1)(a))\alpha_2 -a \cpartial(\alpha_1, \alpha_2))
&=0 .
\end{align*}
Since $\varphi$ is arbitrary, the hypothesis of the corollary implies 
the identity \eqref{immlr}, viz.
\[
\cpartial
(\alpha_1, a\alpha_2)
=
((t\alpha_1)(a))\alpha_2 +a \cpartial(\alpha_1, \alpha_2).
\]
Theorem \ref{immediate4} implies that 
the data $(A,L,\bra,\vartheta)$ constitute a Lie-Rinehart algebra.
In particular, the identity \eqref{immlr} implies that the data satisfy the
axiom \eqref{1.1.b}.
\end{proof}

\begin{thm}
\label{olr}
Suppose that 
the canonical $A$-module morphism 
{\rm \eqref{cano0}}
from $L$ to its double $A$-dual 
$\mathrm{Hom}_A(\mathrm{Hom}_A(L,A),A)$ is an isomorphism
of $A$-modules.
Let 
$d$ 
be an $R$-linear derivation on $\Aalt_A(L,A)$
that turns the graded $A$-algebra
 $\Aalt_A(L,A)$
into a differential graded $R$-algebra.
Then  $(\Aalt_A(L,A),d)$
is the Maurer-Cartan algebra
associated to a (necessarily unique)
 Lie-Rinehart structure on $(A,L)$.
\end{thm}

For example, the hypothesis of Theorem \ref{olr}
as well as that of Corollary \ref{char1}
is satisfied when
$L$ is a finitely generated projective $A$-module.

See \cite{MR2103009} (Lemma 2.2.15) for a traditional proof.
We now sketch a proof in the language and terminology
of the proof of Corollary \ref{char1} above.
An extension of this reasoning yields the proof of a more general result,
Theorem \ref{olr2} below.

\begin{proof}[Proof of Theorem {\rm \ref{olr}}]
Let
\begin{equation}
t\colon sL \longrightarrow \mathrm{Der}(A|R)
\subseteq \mathrm{End}(A,A)
\label{adjoin0}
\end{equation}
be the adjoint 
of the composite
\[
d\colon A \longrightarrow \mathrm{Hom}_A(sL,A)
\subseteq \mathrm{Hom}(sL,A)
\]
of the derivation $d$
with the injection into 
$\mathrm{Hom}(sL,A)$ as displayed.
Then  the derivation
\begin{equation}
\partial^{t}\colon 
\mathrm{Hom}_A(\Lambda[sL],A)
\longrightarrow
\mathrm{Hom}(\Lambda[sL],A)
\end{equation}
is defined, and hence the $R$-module morphism
$\vartheta$, cf.
\eqref{derA}.

Likewise,
 notice that the composite 
\begin{equation}
d\colon 
\mathrm{Hom}_A(sL,A)
\longrightarrow
\mathrm{Hom}_A(\Lambda_2[sL],A)
\subseteq
\mathrm{Hom}(\Lambda_2[sL],A)\cong \Aalt^2(L,A)
\end{equation}
is defined, and
let
\begin{equation}
\widetilde  \partial^{\bra}
=d-\partial^{t}
\colon 
\mathrm{Hom}_A(sL,A)
\longrightarrow
\mathrm{Hom}(\Lambda_2[sL],A) .
\end{equation}
Consider the pairing
\begin{align*}
L \otimes L \otimes \mathrm{Hom}_A(L,A) &\longrightarrow A\\
\alpha\otimes\beta 
\otimes \varphi & \longmapsto \pm
(\widetilde \partial^{\bra}(\varphi\circ s^{-1}))(s\alpha, s\beta ).
\end{align*}
Since the map \eqref{cano0} 
from $L$ to its double $A$-dual is an $A$-module
isomorphism, this pairing induces 
a skew symmetric $R$-linear bracket $\bra$ on $L$, and hence
a coderivation
\[
\cpartial
\colon \Lambda[sL] \longrightarrow \Lambda[sL].
\]

By construction, the pieces of structure 
$(A,L)$,  $\vartheta$, cf. \eqref{derA}, and  $\bra$
satisfy the hypotheses of Corollary \ref{char1}.
Hence the pair $(A,L)$, endowed with $\vartheta$ 
and the bracket $\bra$,
constitutes a Lie-Rinehart algebra.
Still by construction, the differential graded $R$-algebra
$(\Aalt_A(L,A),d)$
is the Maurer-Cartan algebra associated to that Lie-Rinehart algebra.
\end{proof}

\begin{rema}
\label{curious}
Corollary \ref{char1} has the following 
consequence:
If the derivation
${d=\partial^{t} +\apartial}$ on
$\Aalt(L,A)$ restricts to an $R$-linear differential on
$\Aalt_A(L,A)\subseteq \Aalt(L,A)$,
it is necessarily a differential on all of 
$\Aalt(L,A)$.
\end{rema}

\begin{rema} \label{consequences}
Let $\g$ be an ordinary Lie algebra.
The identity \eqref{twc}  
characterizing
a Lie algebra twisting
fixes the operator $\cpartial$ on $\Lambda[s\g]$.
Indeed, the desuspension
$t=s^{-1} \colon s\g \to \g$ 
is the universal Lie algebra twisting cochain
for $\g$.  Let
$x,y\in\g$.
Since 
$\dho t=t\cpartial$ and since
\[
(t\wedge t)(s x,s y)=\tfrac 12 \bra \circ t\otimes t(s x,s y)
= [y,x]
\]
we find
\begin{align}
t\cpartial(s x,s y)&=-t\wedge t (s x,s y)=[x,y]
\\
\cpartial(s x,s y)&=s[x,y].
\end{align}
\end{rema}

\begin{rema}\label{derham}
Let $M$ be a smooth manifold.
In the standard formalism
the de Rham differential $d$ is given by the formulas
\begin{align}
df(X)&=X(f)
\label{standard1}
\\
d\alpha(X,Y)&=X\alpha(Y)-Y\alpha(X)-\alpha[X,Y]
\label{standard2}
\end{align}
etc.
Here $f$ is a smooth function on $M$, $X$ and $Y$ are smooth vector fields,
and $\alpha$ is a smooth 1-form.
While in degree 1, the sign of \eqref{standard1}
is the same as that of the corresponding operator
$\partial^t$ ($=\partial^t + \cpartial$), cf. \eqref{dt},
in degree 1, the sign of \eqref{standard2}
is opposite to that of $\partial^t + \cpartial$.
The sign in \eqref{standard2} arises by abstraction from
the naive evaluation of a 2-form of the kind $dfdh$
on a pair of vector fields by means of
the product formula
\[
\alpha\wedge \beta(X,Y)= \alpha(X)\beta(Y)-\beta(X)\alpha(Y)
\]
involving two 1-forms.
However when we systematically exploit the
Eilenberg-Koszul convention, the product formula
takes the form
\[
\alpha\wedge \beta(X,Y)= -\alpha(X)\beta(Y)+\beta(X)\alpha(Y)
\]
and accordingly, the resulting sign coincides with that of the operator
 $\partial^t + \cpartial$.
\end{rema}

\section{Some technicalities}
We work over a commutative ground ring $R$ 
that contains the rational numbers as a subring. 
Graded objects are graded over the integers.
We understand a differential
as an operator that lowers degree by $1$, and we then indicate the 
degree by subscripts if need be. At times it 
is convenient to switch notationally to superscripts;
here our convention is $A^q=A_{-q}$, so that the differential
takes the form $d\colon A^q \to A^{q+1}$.
Henceforth \lq graded\rq\ means {\em externally graded\/},
cf. e.g., \cite{maclaboo} (p. 175 ff.), that is, we work only with
homogeneous constituents.

\subsection{Hom-differential}
Given two chain complexes $C_1$ and $C_2$, we
write the Hom-differential on $\mathrm{Hom}(C_1,C_2)$
as $\dho$; this differential turns
 $\mathrm{Hom}(C_1,C_2)$ into a 
 chain complex.
We recall that, given a homogeneous member $f$ of  $\mathrm{Hom}(C_1,C_2)$,
the value  $\dho(f)$ is given by
\begin{equation}
\dho(f) =d \circ f + (-1)^{|f|+1} f \circ d.
\label{homd}
\end{equation}

\subsection{Hopf algebra of forms and beyond}
\label{hopfandbeyond}

Later in the paper we will develop an sh Lie-Rinehart
concept that involves sh Lie-algebras.
An sh Lie algebra structure is characterized 
in terms of a cofree differential graded cocommutative coalgebra.
Under our circumstances, the ground ring is assumed to contain the rational 
numbers as a subring, and to develop the sh Lie-Rinehart concept
later in the paper we must express things in terms of the
coalgebra that underlies the symmetric Hopf algebra. We now
explain the requisite technical details.

Let $V$ be a graded $R$-module
and let $\Sigm[V]$  denote the graded symmetric $R$-algebra
on $V$.
The diagonal map of $V$ induces a graded cocommutative
diagonal map on $\Sigm[V]$, indeed,
the familiar shuffle diagonal map, that turns
$\Sigm[V]$ into a graded bialgebra.
Let $\Tc[V]$ denote the graded tensor coalgebra on $V$ in the category of $R$ modules, and let $\Sigmc[V] \subseteq \Tc[V]$
be the largest graded cocommutative subcoalgebra of  $\Tc[V]$
containing $V$, cf. \cite{MR0436178} (p.~ 338) where this construction
is taken over a field of characteristic different from 2.
The universal property of
$\Tc[V]$ entails the existence of a unique extension
of the identity of $V$ to a morphism 
$\Sigm[V] \to \Tc[V]$ 
of graded coalgebras and, since
$\Sigm[V]$ is cocommutative, the values of
that morphism lie in $\Sigmc[V]$, that is,
the
identity of $V$ induces a
canonical morphism
\begin{equation}
\Sigm[V]
\longrightarrow
\Sigmc[V]
\label{cano1}
\end{equation}
of graded cocommutative coalgebras.
Since $R$ is assumed to contain 
the rational numbers as a subring,
\eqref{cano1} is an isomorphism.
Indeed, in degree $n$,
the composite
\[
\begin{CD}
V^{\otimes n}
@>{\mathrm{mult}}>>
\Sigm^n [V]
@>{\tfrac 1{n!}}>>
\Sigm^n [V]
\end{CD}
\]
of the multiplication map $\mathrm{mult}$
with multiplication by $\tfrac 1{n!}$,
restricted to
$\Sigmc_n [V] \subseteq V^{\otimes n}$,
yields the inverse of \eqref{cano1}.
The coalgebra $\Sigmc[V]$ is
the cofree graded cocommutative coalgebra on $V$
in the category of $R$-modules.
The addition of $V$ induces a graded commutative multiplication
that turns $\Sigmc[V]$ into a graded bialgebra,
and \eqref{cano1}
is an isomorphism of graded bialgebras.
For both  $\Sigm[V]$ and $\Sigmc[V]$, multiplication on $V$ by $-1$
induces an antipode turning  each of $\Sigm[V]$ and $\Sigmc[V]$
into a graded Hopf algebra over $R$.

The dual of \eqref{cano1} has the form
\begin{equation}
\mathrm{Hom}(\Sigmc[V],R)
\longrightarrow
\mathrm{Hom}(\Sigm[V],R),
\label{cano2}
\end{equation}
and $\mathrm{Hom}(\Sigm[V],R)$
is traditionally interpreted as an algebra of multilinear
graded symmetric
forms (alternating when $V$ is concentrated in odd degrees and symmetric
in the usual ungraded sense when $V$ is concentrated in even degrees).
Thus the algebra structure of the algebra
$\mathrm{Hom}(\Sigm[V],R)$
 of multilinear  graded symmetric forms is induced by the
shuffle diagonal on $\Sigm[V]$.

\begin{rema} To the browsing reader, the distinction between
$\Sigm[V]$ and $\Sigmc[V]$ might appear pedantic, so here is one hint at 
the difference between the two:
An $R$-linear map $q\colon \Sigmc_2[V]\to R$
on the degree $2$ constituent of  $\Sigmc[V]$
is a (graded) quadratic form, and the composite with
the canonical map
\[
V\otimes V  \longrightarrow \Sigm_2[V]\longrightarrow \Sigmc_2[V]
\]
is the associated (graded) polar form.
\end{rema}

\begin{rema} \label{dividedpowers}
The graded coalgebra $\Sigmc[V]$
is the graded coalgebra
that underlies the divided power Hopf algebra $\Gamma[V]$ on $V$, cf.
\cite{PD}, \cite{MR0065162} (\S \S 17 and 18).
On the other hand,
 the
assignment to $x\in V$ of $\gamma_k(x)= \frac 1{k!}x^k$ ($k \geq 0$)
turns $\Sigm[V]$ into a divided power Hopf algebra. In terms of these
divided powers,
the diagonal map $\Delta$ of $\Sigm[V]$ is given by the identity
\[
\Delta(\gamma_n(x)) = \sum_{j+k=n}\gamma_j(x)\otimes \gamma_k(x), \ n \geq 1,
\ x \in V,
\]
Now
\eqref{cano1} is an isomorphism of divided power  Hopf algebras.
\end{rema}

In terms of the notation $V^*=\mathrm{Hom}(V,R)$,
when $V$ is of finite type
(finitely generated in each degree),
\eqref{cano2}
can be written as
\begin{equation}
\Sigm[V^*]
\longrightarrow
\Sigmc[V^*]
\label{cano3}
\end{equation}
and is formally exactly of the same kind as \eqref{cano1},
with $V^*$ substituted for $V$.

\subsection{Perturbations}
\label{perturbations}
Let $(C,d)$ be a chain complex.
Recall that a {\em perturbation\/} of $d$
is an operator $\partial$ on $C$ such that
$d+\partial$ is a differential.
When $(C,d)$ is a differential graded coalgebra
with counit $\varepsilon \colon C \to R$, 
a perturbation $\partial$ of $d$ 
that is also a coderivation
is
a {\em coalgebra perturbation\/}.
Likewise when $(A,d)$ is a differential graded algebra, 
a perturbation $\partial$ of $d$ 
that is also a derivation
is
an {\em algebra perturbation\/}.

Let $C$ be a coaugmented differential graded coalgebra,
the coaugmentation being written as $\eta\colon R \to C$,
write the resulting {\em coaugmentation filtration\/} as
\begin{equation}
R=C_0\subseteq C_1 \subseteq \ldots \subseteq C_j \subseteq \ldots,
\label{coaugfilt}
\end{equation}
and suppose that $C$ is {\em cocomplete\/}, that is, $C=\cup C_j$.
It then makes sense to require that
a coalgebra perturbation {\em lowers filtration\/}.
We warn the reader that, to avoid an orgy of notation, 
here the subscripts refer to the filtration degree
and not to the ordinary degree.
Henceforth, whenever we use subscripts of this kind,
it will be clear from the context whether
filtration degree or ordinary degree is intended.

Let $\cpartial$ be a coalgebra perturbation that lowers filtration.
Suppose that $\cpartial$ can be written
in the form
\begin{equation}
\cpartial= \cpartial^1 + \ldots + \cpartial^j + \ldots
\label{cpartial}
\end{equation}
such that $\cpartial^j$, when non-zero, 
lowers filtration by $j$ and not by $j+1$ ($j \geq 1$).
We will then say that $\cpartial$ is a {\em filtered coalgebra perturbation\/}.
The bracket subscript is intended as a mnemonic that, for our purposes,
later in the paper, 
such a perturbation characterizes an sh Lie algebra structure.
In particular,
given a filtered coalgebra perturbation $\cpartial$,
the counit $\varepsilon \colon C \to R$ being a morphism of chain complexes,
for $j \geq 1$, the constituent $\cpartial^j$ vanishes on $C_j$ whence,
since $C$ is cocomplete, the infinite sum \eqref{cpartial}
converges naively in the sense that, applied to a specific element of $C$,
only finitely many terms are non-zero.

Given a
filtration decreasing coderivation
of degree $-1$
of the kind \eqref{cpartial} such that 
 $\cpartial^j$, when non-zero, 
lowers filtration by $j$ and not by $j+1$ ($j \geq 1$), save that 
$\cpartial$ is not required to be a perturbation of the coalgebra differential
$d^0$ of $C$,
we refer to a {\em filtered degree $-1$ filtration decreasing
coderivation\/}.
The wording of Theorem \ref{char2} involves a \ppt;
Theorem \ref{char2} is 
a crucial step for Theorems \ref{olr2}, \ref{multishp}, \ref{shlrchar0}, and
\ref{shlrchar}.
For our purposes, a \ppt\  generalizes a skew-symmetric
bracket which is not assumed to be a Lie bracket
(i.e., does not necessarily satisfy the Jacobi identity).
For later reference, we spell out the following.
\begin{prop}
\label{laterref1}
A \ppt\  $\cpartial$  is
a filtered coalgebra perturbation of $d^0$ if and only if,
for $j \geq 1$,
\begin{align}
d^0\cpartial^{j} + \cpartial^{j} d^0 +
\sum_{k=1}^{j-1} \cpartial^k\cpartial^{j-k}
&=0. 
\label{separate11}
\end{align}
\end{prop}

For the benefit of the reader we note that
\eqref{separate11} reads
\begin{align}
d^0\cpartial^1 + \cpartial^1 d^0 &=0 
\label{benefit1}
\\
d^0\cpartial^2 + \cpartial^2 d^0 +
\cpartial^1\cpartial^1&=0 
\label{benefit2}
\\
d^0\partial^3_{\bra} + \partial^3_{\bra} d^0 +
\cpartial^1\cpartial^2+\cpartial^2\cpartial^1&=0 
\end{align}
etc.

\subsection{Filtered Lie algebra twisting cochains}
\label{liealgtc}

Let $C$ be a cocomplete 
coaugmented differential graded cocommutative coalgebra and
$\g$ a differential graded Lie algebra.
Let
\begin{equation}
t = t_1+t_2 + \ldots \colon C \longrightarrow \g
\label{ltwistingc}
\end{equation}
be a Lie algebra twisting cochain such that, for $j \geq 1$,
the constituent $t_j$, if non-zero, is zero on the constituent
$C_{j-1}$ but not on the constituent $C_j$
of the coaugmentation filtration  of $C$, cf. \eqref{coaugfilt}.
Since $C$ is cocomplete,
the infinite sum \eqref{ltwistingc}
converges naively in the sense that, applied to a specific element of $C$,
only finitely many terms are non-zero.
We will then say that $t$ is a {\em filtered Lie algebra twisting cochain\/}.

\begin{rema}
Write the differential graded Lie algebra
$\mathrm{Hom}(C,\g)$ as $\mathcal L$,
write $\mathcal L^0=\mathcal L$
 and, for $j \geq 0$, let
\[
\mathcal L^{j+1}=\ker(\mathrm{Hom}(C,\g) \longrightarrow \mathrm{Hom}(C_j,\g)).
\]
The coaugmentation filtration of $C$ induces the descending filtration
\begin{equation}
\mathcal L=\mathcal L^0\supseteq \mathcal L^1 \supseteq \ldots \supseteq \ldots
\label{descend}
\end{equation}
of differential graded Lie algebras. The Lie algebra twisting cochain 
\eqref{ltwistingc}
being filtered means that, for $j \geq 1$, the constituent $t_j$ lies in
$\mathcal L^j$ but not in $\mathcal L^{j+1}$.
\end{rema}

We will refer to 
a homogeneous morphism of degree $-1$
of the kind \eqref{ltwistingc}
 that is not necessarily a Lie algebra twisting cochain
as a {\em filtered degree $-1$ morphism\/}.
We need this terminology  to be able to phrase Theorem \ref{char2},
a crucial step for Theorems \ref{olr2}, \ref{multishp}, \ref{shlrchar0}, and
\ref{shlrchar}.

For later reference, we spell out the following.
\begin{prop}
\label{laterref2}
Let $(C,d^0)$ be a coaugmented differential graded
cocommutative coalgebra, and let
$\cpartial$ be a filtered coalgebra perturbation of $d^0$.
A \fdm\  
\[
t=t_1+t_2+ \ldots \colon C \longrightarrow \g
\] 
is a 
Lie algebra twisting cochain
\[
t\colon (C,d^0+\cpartial)  \longrightarrow
\g
\]
if and only if, for $j \geq 1$,
\begin{align}
d_0 t^{j} + t^{j}d^0+
\sum_{k=1}^{j-1}
t^k\cpartial^{j-k}+ 
\sum_{k=1}^{j-1}
t^k\wedge t^{j-k} &=0 .
\label{separate22}
\end{align}
\end{prop}

Explicitly, the identities \eqref{separate22} take the form
\begin{align}
d_0 t^1 + t^1d^0&=0
\label{benefit11}
\\
d_0 t^2 + t^2d^0+t^1\cpartial^1+ t^1 \wedge t^1 &=0
\label{benefit12}
\\
d_0 t^3 + t^3d^0+t^1\cpartial^2+ 
t^2\cpartial^1
+[t^1, t^2] &=0
\\
d_0 t^4 + t^4d^0+t^1\partial^3_{\bra}+ 
t^2\cpartial^2+
t^3\cpartial^1
+[t^1, t^3] +t^2\wedge t^2 &=0,
\end{align}
etc.
For clarity, we note that, as for the
term
$\sum_{k=1}^{j-1}
t^k\wedge t^{j-k}$
in \eqref{separate22},
this identity is a concise version of the two identities
\begin{align}
d_0 t^{2j} + t^{2j}d^0+
\sum_{k=1}^{2j-1}
t^k\cpartial^{2j-k}+ 
\sum_{k=1}^{j-1}
[t^k, t^{2k-1}] +t^j\wedge t^j &=0\ (j \geq 1)
\\
d_0 t^{2j+1} + t^{2j+1}d^0+
\sum_{k=1}^{2j}
t^k\cpartial^{2j+1-k}+ 
\sum_{k=1}^{j}
[t^k, t^{2k-1}]&=0 \ (j \geq 0).
\end{align}

\subsection{Lie-Rinehart structures associated to 
$\mathrm{Hom}(C,A)$}
Let $C$ be differential graded cocommutative coalgebra
and $A$ a differential graded commutative algebra.
The Hom-differential $\dho$
and the cup product
turn $\mathrm{Hom}(C,A)$ into a differential graded commutative algebra.
Likewise
$\mathrm{Hom}(C,\mathrm{Der}(A|R))$
acquires a differential graded $R$-Lie algebra
structure and a
differential graded $\mathrm{Hom}(C,A)$-module structure. Furthermore,
with this structure of mutual interaction, the pair
\begin{equation}
(\mathcal A,\mathcal L) =
\left(\mathrm{Hom}(C,A), \mathrm{Hom}(C,\mathrm{Der}(A|R))\right)
\label{mutual2}
\end{equation}
is a differential graded Lie-Rinehart algebra.

Let $t\in
 \mathrm{Hom}(C,\mathrm{Der}(A|R))$
be homogeneous of degree $-1$, at first not necessarily
a Lie algebra twisting cochain $C \to \mathrm{Der}(A|R)$.
The morphism $t$
determines a derivation
\begin{equation}
\partial^t\colon \mathrm{Hom}(C,A) \longrightarrow \mathrm{Hom}(C,A).
\end{equation}
Indeed,
consider
the  universal differential graded algebra
$\mathrm U_A[\mathrm{Der}(A|R)]$
associated to the differential graded Lie-Rinehart algebra
$(A,\mathrm{Der}(A|R))$,
and let $\bra$ denote
 the ordinary commutator 
bracket of $\mathrm U_A[\mathrm{Der}(A|R)]$.
In terms of this bracket, we write
the action of 
$\mathrm{Der}(A|R)$ on $A$
as the bracket operation
\begin{equation}
\mathrm{Der}(A|R) \otimes A \longrightarrow A,\ 
(\delta,a) \longmapsto [\delta,a].
\label{bracket1}
\end{equation}
This bracket, in turn, induces a bracket pairing
\begin{equation}
\bra\colon 
 \mathrm{Hom}(C,\mathrm{Der}(A|R))
\otimes
 \mathrm{Hom}(C,A)
\longrightarrow
 \mathrm{Hom}(C,A),
\label{bracketpairing}
\end{equation}
and the operator $\partial^t$ is given by the expression
\begin{equation} 
\partial^t(\alpha)=[t,\alpha],\ \alpha \in  \mathrm{Hom}(C,A).
\label{expres1}
\end{equation}

We will now suppose that $C$ is coaugmented.

\begin{prop}
The degree $-1$ morphism $t\colon C \to \mathrm{Der}(A|R)$ 
of the underlying graded $R$-modules
is a Lie algebra twisting cochain
if and only if the derivation 
$\partial^t$ is an algebra perturbation of the Hom-differential
$\dho$ on  $\mathrm{Hom}(C,A)$, that is, if and only if
$\dho+\partial^t$ is a(n algebra) differential on
 $\mathrm{Hom}(C,A)$. 
\end{prop}

\begin{proof} This is a consequence of the identity
\begin{align*}
(\dho \partial^t +\partial^t \dho +  \partial^t\partial^t)(\alpha)
&= \left[\dho t +t\wedge t,\alpha\right],\  \alpha \in \mathcal A=\mathrm{Hom}(C,A).
\end{align*}
\end{proof}

\section{Multi derivation Maurer-Cartan algebras }
\label{maintechnical}

Let $(\AA,\dho_0)$ be a differential 
graded algebra, endowed with a filtration
\begin{equation}
\AA=\AA^0\supseteq \AA^1 \supseteq \ldots \supseteq \ldots
\label{descenda}
\end{equation}
that is compatible with the differential $\dho_0$.
We warn the reader that, to avoid an orgy of notation, 
here the superscripts refer to the filtration degree
and not to the ordinary degree (written in superscripts).
Henceforth, whenever we use superscripts of this kind,
it will be clear from the context whether
filtration degree or ordinary degree is intended.
Let $\dho=\sum_{j\geq 1}\dho_j$
be an algebra perturbation of $\dho_0$
such that, for $j \geq 1$, the derivation
$\dho_j$, when non-zero, 
lowers filtration by $j$ but not by $j+1$ ($j \geq 1$)
in the following sense:
for any $\ell \geq 0$, the derivation $\dho_j$,
restricted to
$\AA^{\ell}$, has the form
\begin{equation}
\dho_j\colon \AA^{\ell} \longrightarrow \AA^{\ell+j} 
\label{lformj}
\end{equation}
but does not factor through $\AA^{\ell+j+1}$ as a composite of the kind
$\AA^{\ell} \to \AA^{\ell+j+1} \subseteq \AA^{\ell+j}$.
Here we implicitly assume that $\sum_{j\geq 1}\dho_j$
converges, either naively in the sense that,
given a homogeneous member $\alpha$ of $\AA$,
only finitely many values $\dho_j(\alpha)$ are non-zero
or, more generally, in this sense:
the filtration is complete,
that is, the canonical map 
$\AA \to \lim (\AA/\AA^j)$
is an isomorphism, and $\sum_{j\geq 1}\dho_j$ converges.
We will then say that
\begin{equation}
\left(\AA, \dho_0, \dho_1, \ldots \right)
\end{equation}
is a {\em \multialgebra\/}.

\begin{rema}
\label{classical1}
Given a filtered algebra of the kind \eqref{descenda}, suppose that 
$\AA$ admits a bigrading $\{\AA^{p,q}\}_{p,q}$
with filtration degree $p\geq 0$ and 
complementary degree $q$; here the meaning of the term
\lq complementary degree\rq\ is that $p+q$ recovers the total degree.
Then a special kind of \multialgebra\  structure on $\AA$ is one of the kind
where
the operators $\dho_j$ ($j \geq 0$) 
take the familiar form
\[
\dho_j\colon \AA^{p,q} \longrightarrow \AA^{p+j,q-j+1}.
\]
We will refer to this kind of \multialgebra\  structure as a {\em bigraded
\multialgebra\  structure\/}.
\end{rema}

\begin{rema}
\label{classical2}
Given a filtered algebra of the kind \eqref{descenda},
let $\mathrm E_0(\AA)$
denote the associated bigraded algebra, with bigrading
\[
\mathrm E_0(\AA)^{p,q}=\mathrm E_0(\AA)_{-q}^p=(\AA^p\big /\AA^{p+1})_{-(p+q)},
\]
filtration degree $p\geq 0 $ and 
complementary degree $q$.
A \multialgebra\  structure $\dho_0, \dho_1, \ldots$ on $\AA$
induces a bigraded \multialgebra\  structure 
$\widehat \dho_0, \widehat \dho_1, \ldots$ on
$\mathrm E_0(\AA)$;
we will refer to this  bigraded \multialgebra\  structure
as the {\em  associated bigraded \multialgebra\  structure\/}.
A bigraded \multialgebra\  is isomorphic to its associated
bigraded \multialgebra\  via the obvious map from the former to the latter.
\end{rema}

Let $C$ be a cocomplete coaugmented differential graded cocommutative coalgebra
and $A$ a differential graded commutative algebra.
We write the differential of $C$ as $d^0$.
Furthermore, let 
$\cpartial$ be a \ppt\  
on $C$
and $t\colon C \to \mathrm{Der}(A|R)$
a \fdm.
We use the notation in \eqref{cpartial}
for the constituents
$\cpartial^j$
and that in \eqref{ltwistingc} for the constituents
$t_j$ of $t$ ($j \geq 1$).
Consider the differential graded algebra $\mathrm{Hom}(C,A)$,
endowed with the Hom-differential $\dho_0$,
as well as the graded Lie algebra
 $\mathrm{Hom}(C, \mathrm{Der}(A|R))$,
endowed with the Hom-differential $\dho_0$, where
the notation is slightly abused.
For $j \geq 1$,
the operator $\partial^{t_j}$, cf. \eqref{expres1}, is 
a derivation of $\mathrm{Hom}(C,A)$, and
the coderivation $\cpartial^j$ of $C$ induces a derivation
$\apartial_j$ of $\mathrm{Hom}(C,A)$;
explicitly, given $\varphi \in \mathrm{Hom}(C,A)$ homogeneous,
\begin{equation}
\apartial_j (\varphi)=
(-1)^{|\varphi|+1} \varphi \circ \cpartial^j.
\label{aexplicit}
\end{equation}
Thus the derivations
\begin{align}
\apartial &=\sum \apartial_j
\\
\partial^{t}&= \sum \partial^{t_j}
\\
\dho_{\bra}&= \dho_0 + \apartial
\\
\dho_j&= \partial^{t_j}+\apartial_j \ (j \geq 1)
\label{dhoaj}
\\
\dho&= \dho_{\bra} + \partial^{t} =\dho_0 + \apartial + \partial^{t}
= \dho_0 + \sum \dho_j
\label{dhoa}
\end{align}
of
$\mathrm{Hom}(C,A)$ are defined.
With a slight abuse of notation, we denote the corresponding operator
on the graded Lie algebra
$\mathrm{Hom}(C, \mathrm{Der}(A|R))$ by
\begin{equation}
\dho_{\bra}= \dho_0 + \apartial\colon
\mathrm{Hom}(C, \mathrm{Der}(A|R))
\longrightarrow
\mathrm{Hom}(C, \mathrm{Der}(A|R))
\end{equation}
as well.
For ease of exposition,
we recollect the following (obvious)
statements in the next proposition. 

\begin{prop}
\label{ease1}
\begin{enumerate}
\item The \ppt\  $\cpartial$ is a perturbation 
of the coalgebra differential $d^0$ on $C$
if and only if
\begin{equation}
[d^0, \cpartial]  + \cpartial \cpartial = 0.
\label{iff1}
\end{equation}
\item
Suppose that  $\cpartial$ is a (coalgebra) perturbation of the differential $d^0$ on $C$. Then
$t$ is a Lie algebra twisting cochain of the kind
$(C,d^0+ \cpartial) \to \mathrm{Der}(A|R)$
if and only if
\begin{equation}
\dho_{\bra} t +t\wedge t=0 \in \mathrm{Hom}(C, \mathrm{Der}(A|R)).
\label{iff2}
\end{equation}
\item
The derivation $\apartial$ is an algebra perturbation of the algebra
differential $\dho_0$ on\linebreak 
$\mathrm{Hom}(C,A)$ if and only if
\begin{equation}
[\dho_0, \apartial] + \apartial \apartial = 0.
\label{iff3}
\end{equation}
\item
The identity {\rm \eqref{iff1}}
implies the identity {\rm \eqref{iff3}}.
Thus when $d^0 +\cpartial$ is a (coalgebra) differential on $C$,
the derivation
$\dho_{\bra}=\dho_0+\apartial$
is an algebra differential on $\mathrm{Hom}(C,A)$.
\item
The system of derivations $\{ \dho_j\}_{j \geq 0}$ turns
 $\mathrm{Hom}(C,A)$ into a \multialgebra, that is, the
derivation
$\dho$, cf. {\rm \eqref{dhoa}},
is an algebra differential on $\mathrm{Hom}(C,A)$,
if and only if
\begin{equation}
[\dho_0, \apartial] +
[\dho_0, \partial^t]  +
[\apartial ,\partial^t]
+\apartial \apartial
 +\partial^t \partial^t=0.
\label{iff4}
\end{equation}
\item
The identities {\rm \eqref{iff1}} and {\eqref{iff2}}
together imply the identity {\rm \eqref{iff4}}.
Thus when $\cpartial$ is a (coalgebra) perturbation on $C$
and 
$t\colon (C,d^0 +\cpartial) \to \mathrm{Der}(A|R)$
a Lie algebra twisting cochain,
the derivation
$\dho$, cf. {\rm \eqref{dhoa}},
is an algebra differential on $\mathrm{Hom}(C,A)$.
\end{enumerate}
\end{prop}

Let $V$ be a differential graded $A$-module.
The $A$-module structure being from the left, $V$ also
acquires an obvious differential graded right $A$-module 
structure---this involves the Eilenberg-Koszul convention---,
and the graded tensor powers $V^{\otimes_A n}$ ($n \geq 1$)
are defined. Accordingly,
let $\Sigm_A[V]$ be the graded symmetric $A$-algebra
on $V$, and endow it with the differential graded
$A$-module structure it acquires in an obvious manner.
The differential graded symmetric $R$-algebra $\Sigm[V]$
is defined in the standard way, and the
canonical map
$\Sigm[V]\to \Sigm_A[V]$
is a morphism of $R$-algebras.
As noted before, the diagonal map $V \to V \oplus V$
induces the standard shuffle diagonal on 
 $\Sigm[V]$ and, likewise,
a  shuffle diagonal on 
$\Sigm_A[V]$ in such a way that
$\Sigm[V]\to \Sigm_A[V]$ is compatible with the diagonals
but, beware,
$\Sigm[V]\to \Sigm_A[V]$
is not in a naive manner a morphism of coalgebras.

We will now apply the previous discussion to $C= \Sigm[V]$ and $A$, as well as,
suitably adjusted, to  $\Sigm_A[V]$ and $A$.
The graded commutative algebra
$\mathrm{Hom}(C,A)$
(endowed with the cup multiplication)
is that of graded symmetric $A$-valued $R$-multilinear forms
on $V$, the differentials on $V$ and $A$ induce the algebra 
differential written before as $\dho_0$, and we continue to use this notation.
Likewise,
$\mathrm{Hom}_A(\Sigm_A[V],A)$
is the graded $A$-module
of graded symmetric $A$-valued graded $A$-multilinear forms
on $V$.
A little thought reveals that,
since $\Sigm_A[V]$ is a differential graded $A$-module,
with a slight abuse of the notation $\dho_0$,
this operator passes to
$\mathrm{Hom}_A(\Sigm_A[V],A)$
and, with respect to the ordinary multiplication of
forms (induced by the shuffle diagonal on $\Sigm_A[V]$
and the multiplication of $A$), turns
$\mathrm{Hom}_A(\Sigm_A[V],A)$
into a differential graded commutative
$R$-algebra. We write this differential graded algebra as
$(\mathrm{Sym}_A(V,A),\dho_0)$.

Recall the coaugmentation filtration
$R=C_0 \subseteq C_1 \subseteq \ldots$ of 
the coaugmented differential graded cocommutative
$R$-coalgebra $C=\Sigm[V]$ and,
for
$j \geq 0$, let
\[
\mathrm{Hom}(\Sigm[V],A)^{j+1}
=\ker(\mathrm{Hom}(C,A) \longrightarrow \mathrm{Hom}(C_j,A)).
\]
Relative to the differential $\dho_0$,
\begin{equation}
\mathrm{Hom}(\Sigm[V],A)=
\mathrm{Hom}(\Sigm[V],A)^0
\supseteq 
\mathrm{Hom}(\Sigm[V],A)^1 \supseteq \ldots \supseteq \ldots
\label{descend2}
\end{equation}
is a descending filtration
of differential graded algebras and,
in an obvious manner, 
still relative to the differential $\dho_0$,
this filtration induces as well a filtration
\begin{equation}
\mathrm{Sym}_A(V,A)=\mathrm{Sym}_A(V,A)^0
\supseteq \mathrm{Sym}_A(V,A)^1 \supseteq \ldots \supseteq 
\mathrm{Sym}_A(V,A)^j \supseteq \ldots 
\label{augmentationfiltsym}
\end{equation}
of differential graded $R$-algebras.

As noted earlier, since the ground ring $R$ contains the rational numbers
as a subring, as a Hopf algebra, the differential graded 
symmetric algebra $\Sigm[V]$
is actually canonically isomorphic to the
cofree differential graded cocommutative
coalgebra $\Sigmc[V]$ on $V$;
in particular, the coaugmentation filtration
coincides with the tensor power filtration.

As before, let  
$\cpartial$ be a \ppt\  
on $C=\Sigm[V]$
and $t\colon C=\Sigm[V] \to \mathrm{Der}(A|R)$
a \fdm.
The following is immediate.

\begin{prop}
\label{immediate2}
Suppose that, for $j \geq 1$, each derivation $\dho_j$ of 
$\mathrm{Hom}(\Sigm[V],A)$ 
(though not necessarily the individual constituents
$\partial^{t_j}$ and
$\apartial_j$ of $\dho_j= \partial^{t_j} +\apartial_j$)
passes to a derivation of
$\mathrm{Sym}_A(V,A)=\mathrm{Hom}_A(\Sigm_A[V],A)$,
necessarily
lowering the filtration {\rm \eqref{augmentationfiltsym}}  
by $j$
(with a slight abuse of the notation
$\dho_j$). 
When $\dho$ is an algebra differential---equivalently,
when $\sum_{j \geq 1}\dho_j$, cf. {\rm \eqref{dhoa}}, 
is an algebra perturbation
of the algebra differential $\dho_0$ on  
$\mathrm{Sym}_A(V,A)$---, the data
\begin{equation}
(\mathrm{Sym}_A(V,A),\dho_0, \dho_1, \dho_2, \ldots)
\label{multisym}
\end{equation}
constitute a \multialgebra.
\end{prop}

\begin{cor}
\label{cruc2}
Suppose that $\cpartial$ is a (coalgebra) perturbation on $C=\Sigm[V]$
and $t$ a filtered Lie algebra twisting cochain 
$(C,d^0 +\cpartial) \to \mathrm{Der}(A|R)$.
Suppose that, furthermore, for $j \geq 1$, each derivation $\dho_j$ of 
$\mathrm{Hom}(\Sigm[V],A)$ 
(though not necessarily the individual constituents
$\partial^{t_j}$ and
$\apartial_j$ of $\dho_j= \partial^{t_j} +\apartial_j$)
passes to a derivation of
$\mathrm{Sym}_A(V,A)=\mathrm{Hom}_A(\Sigm_A[V],A)$.
Then {\rm \eqref{multisym}} is a \multialgebra.
\end{cor}

\begin{rema}
\label{classical}
The graded commutative algebras 
$\mathrm{Hom}(\Sigm[V],A)$ and $\mathrm{Hom}_A(\Sigm_A[V],A)$
are bigraded in an obvious manner:
Write the grading of $A$ in superscripts, so that
$A_{-q}=A^{q}$ and so that the differential
of $A$ takes the form
$d_0\colon A^q \to A^{q+1}$.
A homogeneous member of
$\mathrm{Hom}(\Sigm[V],A)$
of the kind
\[
f\colon V_{j_1} \times \ldots \times V_{j_p}\to A^{p+q-j_1-\ldots-j_p}
\]
has filtration degree $p$ and complementary degree $q$,
cf. Remark \ref{classical1}, and the resulting bigrading
of $\mathrm{Hom}(\Sigm[V],A)$ passes to $\mathrm{Hom}_A(\Sigm_A[V],A)$.
\end{rema}

We will now
explore the question to what extent the converse of 
the statement of Corollary \ref{cruc2} holds.
Here is a generalization of 
the sufficiency claim of
Corollary \ref{char1} above.

\begin{thm}
\label{char2}
Given the \ppt\ 
$\cpartial$ 
and the \fdm\ $t$,
suppose that, for $j \geq 1$, each derivation $\dho_j$ of 
$\mathrm{Hom}(\Sigm[V],A)$ 
(though not necessarily the individual constituents
$\partial^{t_j}$ and
$\apartial_j$ of $\dho_j= \partial^{t_j} +\apartial_j$)
passes to a derivation of
$\mathrm{Sym}_A(V,A)=\mathrm{Hom}_A(\Sigm_A[V],A)$ 
and that
{\rm \eqref{multisym}} is a \multialgebra.
Suppose, furthermore, that the canonical morphism
\begin{equation}
V
\longrightarrow
\mathrm{Hom}_A(V,\mathrm{Hom}_A(V,A),A)
\label{cano}
\end{equation}
of graded $A$-modules
(from $V$ into its double $A$-dual)
is injective.
Then $\cpartial$ is a coalgebra perturbation of 
the coalgebra differential $d^0$ of
$C=\Sigm[V]$,
and $t$ is a Lie algebra twisting cochain 
$(C,d^0 +\cpartial) \to \mathrm{Der}(A|R)$.
\end{thm}

\begin{rema}
In Theorem \ref{char2},
we do not assume that
the \multialgebra\ 
\eqref{multisym}
comes from a \multialgebra\  structure on the ambient algebra
$\mathrm{Sym}(V,A)
=\mathrm{Hom}(\Sigm[V],A)$.
\end{rema}

However, Theorem \ref{char2} has the following curious consequence,
cf. Remark \ref{curious} above:

\begin{cor}
\label{curious2}
Under the circumstances of Theorem {\rm \ref{char2}}, the data
\begin{equation}
(\mathrm{Sym}(V,A),\dho_0, \dho_1, \dho_2, \ldots)
=
(\mathrm{Hom}(\Sigm[V],A),\dho_0, \dho_1, \dho_2, \ldots)
\label{necessarily1}
\end{equation}
necessarily constitute a \multialgebra\  as well.
\end{cor}

\begin{proof}[Proof of Theorem {\rm \ref{char2}}]
We must confirm the identities \eqref{separate11}
and \eqref{separate22}.

We note first that, on $\mathrm{Hom}(\Sigm[V],A)$,
\begin{equation}
\begin{aligned}
\sum_{k=0}^j \dho_k \dho_{j-k}
&=
\left([\dho_0,\partial^{t_{j}}]  
+ \sum_{k=1}^{j-1}[\apartial_k, \partial^{t_{j-k}}] 
+ \sum_{k=1}^{j-1}\partial^{t_k}\partial^{t_{j-k}}\right)
\\
&\quad
+[\dho_0,\apartial_j]+
\sum_{k=1}^{j-1} \apartial_k\apartial_{j-k}.
\end{aligned}
\label{trick1}
\end{equation}

Let $\ell \geq0$, $j \geq 1$, and let 
$a \in A$.
Exploiting the bracket pairing \eqref{bracketpairing}, we find
\begin{equation}
\begin{aligned}
{}&\left([\dho_0,\partial^{t_{j}}]  
+ \sum_{k=1}^{j-1}[\apartial_k, \partial^{t_{j-k}}] 
+ \sum_{k=1}^{j-1}\partial^{t_k}\partial^{t_{j-k}}\right)
a
\\
&=
\left[d_0 t^{j} + t^{j}d^0+
\sum_{k=1}^{j-1}
t^k\cpartial^{j-k}+ 
\sum_{k=1}^{j-1}
t^k\wedge t^{j-k},a\right].
\end{aligned}
\label{exploit1}
\end{equation}
Since $a$ and $j$ are arbitrary, 
and since the
second constituent
\[
[\dho_0,\apartial_j]+
\sum_{k=1}^{j-1} \apartial_k\apartial_{j-k}
\]
on the right-hand side of \eqref{trick1},
evaluated at $a\in A$ is zero,
we conclude that
the identities \eqref{separate22} hold.

To establish the identities \eqref{separate11}, 
let again $j \geq 1$.
We note first that, $j$ having been fixed,
the identity \eqref{separate22}, in turn, implies that
the identity
\begin{equation}
[\dho_0,\partial^{t_{j}}]  
+ \sum_{k=1}^{j-1}[\apartial_k, \partial^{t_{j-k}}] 
+ \sum_{k=1}^{j-1}\partial^{t_k}\partial^{t_{j-k}}=0
\label{allofh}
\end{equation}
holds on $\mathrm{Hom}(\Sigm[V],A)$
(not just on $\mathrm{Hom}_A(\Sigm[V],A)$, in fact
it would not even make sense  on $\mathrm{Hom}_A(\Sigm[V],A)$).

Let $x \in \Sigm^{j+1}[V]$ homogeneous
(beware, $x$ is not taken to be a member of  $\Sigm_A^{j+1}[V]$)
and $\varphi \in \mathrm{Hom}_A(V,\Sigm^{j+1}_A[V])$ homogeneous. 
For the moment, view $\varphi$ as a member of  $\mathrm{Hom}(V,\Sigm^{j+1}_A[V])$. By construction
\begin{align*}
\left([\dho_0,\apartial_j]+
\sum_{k=1}^{j-1} \apartial_k\apartial_{j-k}\right)\varphi(x)
&=
\pm\varphi\left(
\left(
d^0\cpartial^{j} + \cpartial^{j} d^0 +
\sum_{k=1}^{j-1} \cpartial^k\cpartial^{j-k}
\right)(x)\right)
\end{align*}
However, in view of \eqref{trick1} and \eqref{allofh},
\begin{equation}
\left([\dho_0,\apartial_j]+
\sum_{k=1}^{j-1} \apartial_k\apartial_{j-k}\right)\varphi
=
\left(\sum_{k=0}^j \dho_k \dho_{j-k}\right)\varphi.
\end{equation}
Now
\begin{equation*}
\left(
d^0\cpartial^{j} + \cpartial^{j} d^0 +
\sum_{k=1}^{j-1} \cpartial^k\cpartial^{j-k}
\right)(x)\in V.
\end{equation*}
By assumption,
$\sum_{k=0}^j \dho_k \dho_{j-k}$ is zero on
$\mathrm{Sym}_A(V,A)$ and
$\varphi$ was taken in $\mathrm{Sym}_A(V,A)$ whence
\begin{equation}
\varphi\left(\left(
d^0\cpartial^{j} + \cpartial^{j} d^0 +
\sum_{k=1}^{j-1} \cpartial^k\cpartial^{j-k}
\right)(x)\right)=0.
\end{equation}
Since the canonical morphism \eqref{cano} of graded $A$-modules
is supposed to be injective and since $\varphi$ is arbitrary,
we conclude
\begin{equation}
\left(
d^0\cpartial^{j} + \cpartial^{j} d^0 +
\sum_{k=1}^{j-1} \cpartial^k\cpartial^{j-k}
\right)(x)=0.
\end{equation}
Since $x\in \Sigm^{j+1}[V]$ is arbitrary, we find
\begin{equation}
d^0\cpartial^{j} + \cpartial^{j} d^0 +
\sum_{k=1}^{j-1} \cpartial^k\cpartial^{j-k}=0,
\end{equation}
that is, the identity \eqref{separate11} holds as well.
\end{proof}

\begin{rema}
\label{instructive}
It is instructive to illustrate the reasoning in the above proof
in low degrees:

For $j=1$, the identity \eqref{trick1} reads
\[
\dho_0 \dho_1+ \dho_1 \dho_0 =[\dho_0, \partial^{t_1}]+[\dho_0, \apartial_1],
\]
the hypothesis $0= 
\dho_0 \dho_1+ \dho_1 \dho_0$  
comes down to
\begin{align}
0&=[\dho_0, \partial^{t_1}]+[\dho_0, \apartial_1] 
\colon 
\mathrm{Hom}_A(\Sigm_A^{\ell}[V],A) \to
\mathrm{Hom}_A(\Sigm_A^{\ell+1}[V],A),
\end{align}
and the identity \eqref{exploit1} amounts to
\[
[\dho_0, \partial^{t_1}]a
=\left[d_0t_1 +t_1d^0,a\right],\ a \in A.
\]
Since $a \in A$ is arbitrary,
we conclude that 
$t_1$ satisfies the identity \eqref{benefit11},
viz. $d_0t_1 +t_1d^0=0$, whence
$t_1\colon V \to A$
is a morphism of chain complexes.
Consequently the derivation
$
[\dho_0, \partial^{t_1}]
$
is zero on all of $\mathrm{Hom}(\Sigm[V],A)$.
Since the sum $[\dho_0, \partial^{t_1}]+[\dho_0, \apartial_1]$
is zero, 
in view of
\[
\left([\dho_0, \partial^{t_1}]+[\dho_0, \apartial_1]\right)\varphi(x)
=
\left[d_0t_1 +t_1d^0,\varphi\right](x)
\pm \varphi((\cpartial^1 d^0 + d^0\cpartial^1)(x)),
\]
for any $\varphi \in \mathrm{Hom}_A(V,A)$ and any $x \in \Sigm^2[V]$,
we deduce  the identity \eqref{benefit1}, viz.
\[
d^0\cpartial^1+\cpartial^1 d^0=0.
\]
Consequently the derivation
$
[\dho_0, \apartial_1]
$
is zero on all of $\mathrm{Hom}(\Sigm[V],A)$.

For $j=2$, the hypothesis 
$\dho_0 \dho_2+ \dho_2 \dho_0+\dho_1\dho_1 =0$ implies
\begin{align}
[\dho_0, \partial^{t_2}]+ 
[\dho_0, \apartial_2]+ 
[\partial^{t_1}, \apartial_1]+
\partial^{t_1}\partial^{t_1}+ \apartial_1 \apartial_1 .
&=0
\end{align}
Given $a \in A$, exploiting the identity \eqref{trick1}, we find that
\begin{align*}
\left([\dho_0, \partial^{t_2}]+  
\apartial_1\partial^{t_1}+
\partial^{t_1}\partial^{t_1}\right)(a)
&= \left[d_0t_2 +t_2d^0+t_1\cpartial^1+t_1\wedge t_1,a\right]
\end{align*}
is zero. Since $a\in A$ is arbitrary, we conclude 
that
$d_0t_2 +t_2d^0+t_1\cpartial^1+t_1\wedge t_1$ is zero,
that is, $t_1$ and $t_2$ satisfy the identity \eqref{benefit12}.
Consequently
$[\dho_0, \partial^{t_2}]+  
\apartial_1\partial^{t_1}+
\partial^{t_1}\partial^{t_1}=0$, whence
$
[\dho_0, \apartial_2]+  \apartial_1 \apartial_1=0$,
and thence \eqref{benefit2} holds, viz.
\[
d^0\cpartial^2 +\cpartial^2d^0  +\cpartial^1 \cpartial^1=0.
\]
\end{rema}

The next result extends Theorem \ref{olr} above.

\begin{thm}
\label{olr2}
Suppose that the canonical morphism
{\rm\eqref{cano}}
of graded $A$-modules
is an isomorphism.
Let $\{\dho_j\}_{j \geq 1}$
be a family of derivations of
$\mathrm{Sym}_A(V,A)$ such that
$\dho_j$ lowers the filtration
{\rm \eqref{augmentationfiltsym}}
by $j$. Suppose, furthermore, that
\begin{equation}
(\mathrm{Sym}_A(V,A),\dho_0, \dho_1, \dho_2, \ldots)
\label{multisym2}
\end{equation}
is a \multialgebra.
Then  the family $\{\dho_j\}_{j \geq 1}$ arises from a
(necessarily unique) filtered  (coalgebra) perturbation 
$\cpartial$ of the coalgebra differential $d^0$ 
on $C=\Sigm[V]$
and a filtered Lie algebra twisting cochain 
$t\colon (C,d^0 +\cpartial) \to \mathrm{Der}(A|R)$.
\end{thm}

\begin{proof}
For $j \geq 1$, 
let
\begin{equation}
t_j\colon \Sigm^j[V] \longrightarrow \mathrm{Der}(A|R)
\subseteq \mathrm{End}(A,A)
\label{adjoin1}
\end{equation}
be the adjoint 
of the composite
\[
\dho_j\colon A \longrightarrow \mathrm{Hom}_A(\Sigm_A^j[V],A)
\subseteq \mathrm{Hom}(\Sigm_A^j[V],A)
\]
of the derivation $\dho_j$
with the injection into 
$\mathrm{Hom}(\Sigm_A^j[V],A)$ as displayed.
This yields a \fdm\  
$t \colon \Sigm[V] \to\mathrm{Der}(A|R) $.
By construction, the derivations
\begin{equation}
\partial^{t_j}\colon 
\mathrm{Hom}_A(\Sigm[V],A)
\longrightarrow
\mathrm{Hom}(\Sigm[V],A)\ (j \geq 1)
\end{equation}
are then defined.

Likewise,
let $j \geq 1$, notice that the composite 
\begin{equation}
\dho_j\colon 
\mathrm{Hom}_A(V,A)
\longrightarrow
\mathrm{Hom}_A(\Sigm_A^{j+1}[V],A)
\subseteq
\mathrm{Hom}(\Sigm^{j+1}[V],A)
\end{equation}
is defined, and
let
\begin{equation}
\widetilde  \partial^{\bra}_j
=\dho_j-\partial^{t_j}
\colon 
\mathrm{Hom}_A(V,A)
\longrightarrow
\mathrm{Hom}(\Sigm^{j+1}[V],A) .
\end{equation}
Consider the pairing
\begin{align*}
V^{\otimes (j+1)} \otimes \mathrm{Hom}_A(V,A) &\longrightarrow A\\
\alpha_1 \otimes\ldots \otimes \alpha_{j+1} 
\otimes \varphi & \longmapsto \widetilde \partial^{\bra}_j\varphi(
\alpha_1, \ldots, \alpha_{j+1} ).
\end{align*}
Since the map from $V$ to its double $A$-dual is an $A$-module
isomorphism, this pairing induces 
an operation 
\[
\widetilde \partial_{\bra}^j
\colon \Sigm^{j+1}[V] \longrightarrow V,
\]
and we extend $\widetilde \partial_{\bra}^j$ to a coderivation
\[
\cpartial^j
\colon \Sigm[V] \longrightarrow \Sigm[V].
\]
This yields a \ppt\ 
$\cpartial= \sum_{j\geq}\cpartial^j$. 

By construction,
the hypotheses of Theorem \ref{char2} are satisfied.
In view of that theorem,
$\cpartial$ is a filtered coalgebra perturbation of the coalgebra
differential $d^0$ of $C=\Sigm[V]$,
and $t$ is a filtered Lie algebra twisting cochain 
$(C,d^0 +\cpartial) \to \mathrm{Der}(A|R)$.
By construction, the
family $\{\dho_j\}_{j \geq 1}$ 
of derivations
arises from these data as asserted.
\end{proof}

Under the circumstances of Theorem \ref{char2},
when each derivation $\dho_j$ is $A$-multilinear
in such a way that
\eqref{multisym} is a \multialgebra,
we refer to the \multialgebra\  as the {\em \mmd\  algebra\/}
associated to the data.
Theorem \eqref{olr2} says that, 
when \eqref{cano} is an isomorphism,
any \multialgebra\  structure on
$\mathrm{Sym}_A(V,A)$
is the \mmd\  algebra
associated to a filtered
coalgebra perturbation $\cpartial$ on $\Sigm[V]$ and
a filtered Lie algebra twisting cochain 
$t\colon (\Sigm[V],d^0+\cpartial) \to \mathrm{Der}(A|R)$
having the property 
that each derivation of the kind $\dho_j$, cf. \eqref{dhoaj}, 
is $A$-multilinear.

\section{sh Lie-Rinehart algebras}
\label{sh}
\subsection{sh Lie algebras}

Let $\g$ be an $R$-chain complex.
An {\em sh-Lie algebra\/} or, equivalently, 
$L_{\infty}$-{\em algebra structure\/}
on $\g$ is a 
coaugmentation filtration lowering
coalgebra perturbation
$\cpartial\colon \Sigmc[s\g] \to \Sigmc[s\g]$
of the coalgebra differential $d^0$ on  $\Sigmc[s\g]$.
We will then refer to the pair
$(\g,\cpartial)$ as an {\em sh Lie algebra\/}.

In the literature, it is common to write the structure
in terms of higher order brackets. 
While we do not use the bracket formalism in the paper,
for the benefit of the reader, we now explain
how the higher order brackets arise:
For $n \geq 2$, consider the graded symmetrization map
\[
\begin{CD}
(s\g)^{\otimes n}
@>{\mathrm{sym}}>>
\Sigmc_n [s\g],\ 
a_1\otimes \ldots\otimes  a_n \longmapsto \tfrac 1{n!} 
\sum \pm a_{\sigma 1}\otimes \ldots a_{\sigma n},
\end{CD}
\]
and use the bracket notation
\begin{equation}
\begin{CD}
\bran\colon
\g^{\otimes n} 
@>{s^{\otimes n}}>>
(s\g)^{\otimes n}
@>{\mathrm{sym}}>>
\Sigmc_n [s\g]
@>{\cpartial^{n-1}}>>
\Sigmc_1 [s\g]
@>{\tau_{\g}}>>
\g
\end{CD}
\label{brannn}
\end{equation}
for the depicted $\g$-valued operation of $n$-variables
ranging over $\g$; by construction,
the operation $\bran$
has homogeneous  degree $n-2$ and is
graded skew symmetric.

Let $(\g,\cpartial)$
be an
$L_{\infty}$-algebra
and suppose that $\g$ is concentrated in degrees $\leq 0$;
we then write
$\g^j=\g_{-j}$ ($j \geq 0$). The $L_{\infty}$-structure
is given by a system of bracket operations
\begin{equation}
\bran_n\colon  \g^{j_1} \times \ldots  \times \g^{j_n}
\longrightarrow  \g^{j_1+\ldots +j_n-n+2}
\end{equation} 
Thus
\begin{align}
\bra_2&\colon  \g^{j_1} \times \g^{j_2}
\longrightarrow  \g^{j_1+j_2}
\\
\braaa&\colon  \g^{j_1} \times \g^{j_2}\times \g^{j_3}
\longrightarrow  \g^{j_1+j_2+j_3-1}
\\
\braaa&\colon  \g^{1} \times \g^{0}\times \g^{0}
\longrightarrow  \g^{0}
\end{align} 
etc.

\subsection{sh Lie algebra action by derivations}

Let $A$ be a differential graded commutative $R$-algebra.
Given an sh Lie algebra
 $(\g,\cpartial)$, we define an {\em sh-action of\/}
$(\g,\cpartial)$ {\em  on\/} $A$ {\em by derivations\/} 
to be a Lie algebra twisting cochain
\begin{equation}
t\colon
(\Sigmc[s\g],d^0+\cpartial)
\longrightarrow
\mathrm{Der}(A|R).
\label{tc}
\end{equation}

Let $t\colon \Sigmc[s\g] \to \mathrm{Der}(A|R)$
and $t'\colon \Sigmc[s\g'] \to \mathrm{Der}(A'|R)$
be two sh actions by derivations and
$\varphi \colon A \to A'$ a morphism
of differential graded algebras. 

\begin{defi}\label{mor1}
Given a morphism
\[
\Phi \colon (\Sigmc[s \g],d^0+\cpartial) \longrightarrow (\Sigmc[s\g'],d^0+\cpartial') 
\]
of differential graded coalgebras,
the pair
\begin{equation}
(\varphi,\Phi)\colon  (A,\g,\cpartial,t)\longrightarrow (A',\g',\cpartial',t')
\end{equation}
is a {\em morphism of sh actions by derivations\/}
when  
the adjoints
$t^{\sharp}\colon \Sigmc[s\g] \otimes A \to A$
and
$(t')^{\sharp}\colon \Sigmc[s\g'] \otimes A' \to A'$
of $t$ and $t'$, respectively, make the diagram
\begin{equation}
\begin{CD}
\Sigmc[s \g] \otimes A @>{t^{\sharp}}>> A
\\
@V{\Phi \otimes \varphi}VV
@V{\varphi}VV
\\
\Sigmc[s\g'] \otimes A' @>{(t')^{\sharp}}>> A'
\end{CD}
\end{equation}
commutative.
\end{defi}

For reasons of variance,  a
general morphism of sh actions does not induce a morphism
between the associated Maurer-Cartan algebras, see Remark \eqref{variance} 
below. In ordinary Lie algebra cohomology theory,
one takes care of the variance problem by means of comorphisms.
To extend the comorphism concept to the present situation, define
\begin{equation}
t^{\flat}\colon \Sigmc[s\g] \longrightarrow \mathrm{Der}(A,A')
\end{equation}
by
$t^{\flat}(x) =\varphi \circ (t(x)),\ x \in \Sigm[s\g]$ and
\begin{equation}
\varphi^{\flat}\colon \mathrm{Der}(A'|R)\longrightarrow \mathrm{Der}(A,A')
\end{equation}
by $\varphi^{\flat}(\delta)=\delta \circ \varphi$.

\begin{defi}\label{mor2}
Given a morphism
\[
 \Phi \colon (\Sigmc[s\g'],d^0+\cpartial') \longrightarrow (\Sigmc[s \g],d^0+\cpartial) 
\]
of differential graded coalgebras,
the pair
\begin{equation}
(\varphi,\Phi)\colon  (A',\g',\cpartial',t')\longrightarrow 
(A,\g,\cpartial,t)
\end{equation}
is a {\em comorphism of sh actions by derivations\/}
when  
the diagram
\begin{equation}
\begin{CD}
\Sigmc[s \g'] @>{\Phi}>> \Sigmc[s\g]
\\
@V{t'}VV
@V{t^{\flat}}VV
\\
 \mathrm{Der}(A'|R) @>{\varphi^{\flat}}>>\mathrm{Der}(A, A') 
\end{CD}
\label{shder}
\end{equation}
is commutative. 
\end{defi}
Notice that, between the two definitions \eqref{mor1} and \eqref{mor2},
there is a difference of variance. Notice also
when $\g$ and $\g'$ are ordinary Lie algebras
and $t$ and $t'$ come from ordinary Lie algebra actions by derivations,
in terms of the bracket notation $\bra\colon \g \times A \to A$
and  ${\bra\colon \g' \times A' \to A'}$
for these actions,
the commutativity of \eqref{shder} comes down to the familiar identity
\begin{equation}
\varphi[\Phi(x),a] =[x,\varphi(a)],\ x\in \g',\ a\in A.
\end{equation}
This identity says that $\varphi$ is a morphism
of $\g'$-modules when $\g'$ acts on $A$ through $\Phi\colon \g' \to \g$. 
The following proposition generalizes a classical 
observation in ordinary Lie algebra cohomology theory.

\begin{prop}
\label{como1}
A comorphism
$(\varphi,\Phi)\colon  (A',\g',\cpartial',t')\to 
(A,\g,\cpartial,t)$
of sh actions by derivations induces a morphism
\begin{equation}
(\varphi_*,\Phi^*)\colon (\mathrm{Hom}(\Sigmc[s\g],A),\dho_0+\apartial +\partial^t)
\longrightarrow
(\mathrm{Hom}(\Sigmc[s\g'],A'),\dho_0+(\apartial)' +\partial^{t'})
\label{como22}
\end{equation}
of differential graded $R$-algebras.
\end{prop}

\begin{rema}
\label{variance}
A comorphism
$(\varphi,\Phi)\colon  (A',\g',\cpartial',t')\to 
(A,\g,\cpartial,t)$
of sh actions by derivations
having $A=A'$ and $\varphi=\mathrm{Id}$ is simply a morphism
\[
(\mathrm{Id},\Phi)\colon  (A,\g',\cpartial',t')\to 
(A,\g,\cpartial,t)
\]
of sh actions by derivations.
However, for reasons of variance,
a general morphism of sh actions by derivations
cannot induce a morphism of the kind \eqref{como22}.
\end{rema}

An sh Lie algebra action may be written in terms
of bracket operations in the following manner;
again we do not need these brackets but spell them out
for the benefit of the reader.
Let $t$
be an sh-action of the kind \eqref{tc} of $(\g,\cpartial)$ on $A$ by derivations.
For $n \geq 1$,
consider the composite
\begin{equation}
\begin{CD}
\g^{\otimes n} 
@>{s^{\otimes n}}>>
(s\g)^{\otimes n}
@>{\mathrm{sym}}>>
\Sigmc_n [s\g]
@>{t_n}>>
\mathrm{Der}(A|R),
\end{CD}
\label{bran1}
\end{equation}
and write its adjoint as an $A$-valued operation
\begin{equation}
\bracn\colon
\g^{\otimes n} 
\otimes A
\longrightarrow
A
\label{adjoint}
\end{equation}
having $n$ arguments from $\g$ and one argument from $A$.
Given homogeneous $x_1,\ldots,x_n \in \g$, the operation
\[
\{x_1,\ldots,x_n|\,\cdot\,\} \colon A \longrightarrow A
\]
is a homogeneous derivation of $A$ of degree
$|x_1|+\ldots +|x_n|+n-1$.

\subsection{sh Lie-Rinehart structure}
\label{shlr}
In terms of the formalism so far developed,
in the spirit of \cite{MR1854642} (Def. 4.9 p. 157),
we will now propose  
a definition of
an sh Lie-Rinehart algebra.

Let $A$ be a differential graded commutative $R$-algebra
and $(L,\cpartial)$ an sh Lie algebra over $R$.
In view of the isomorphism \eqref{cano1}, we will henceforth identify
$\Sigmc[sL]$ with $\Sigm[sL]$, endowed with the graded shuffle diagonal.
Let $t \colon \Sigm[sL]\to \mathrm{Der}(A|R)$
be an sh-action of $L$ on $A$ by derivations,
and suppose that $L$ carries a
differential graded $A$-module structure.
Take $t=t_1+t_2+\ldots$ to be a filtered Lie algebra twisting 
cochain in the obvious manner, that is, let $t_j$ ($j \geq 1$)
be the component of $t$
defined on the $j$th graded symmetric power $\Sigm^j[sL]$ of the suspension
$sL$ of $L$. Likewise take $\cpartial=\cpartial^1+ \cpartial^2 + \ldots$ 
to be a filtered coalgebra perturbation
in the obvious way, that is, for $j \geq 1$, let $\cpartial^j$
denote the coderivation determined by the constituent
\[ 
\cpartial\colon \Sigm^{j+1}[sL] \longrightarrow sL
\]
of $\cpartial$.
Endow the suspension $sL$ of $L$ with the induced
differential graded $A$-module structure.

\begin{defi}\label{shlrdefi}
The data $(A,L,\cpartial,t)$
constitute an {\em sh Lie-Rinehart algebra\/}
when
$t$ is graded $A$-multilinear and
the data satisfy the axiom
\eqref{sh12} below (for $j \geq 1$):
\begin{equation}
\begin{aligned}
\cpartial^j(\alpha_1,\ldots,\alpha_j, a \alpha_{j+1})
&=
 t_j(\alpha_1,\ldots,\alpha_j)(a) \alpha_{j+1}
\\
&\quad
+(-1)^{(|\alpha_1|+\ldots +|\alpha_j|+1)|a|}
a\cpartial^j(\alpha_1,\ldots,\alpha_j, \alpha_{j+1})
\label{sh12}
\end{aligned}
\end{equation}
where $\alpha_1,\ldots,\alpha_j, \alpha_{j+1}$ are homogeneous members of $sL$ and $a$
is a homogeneous member of  $A$.
\end{defi}

\begin{rema}
\label{notice1}
For $j \geq 1$, the condition \eqref{sh12}
measures
 the deviation of  $\cpartial^j$ 
from being graded $A$-multilinear.
\end{rema}

\begin{rema} \label{grasp}
In \cite{MR1854642} (Def. 4.9 p. 157),
the terminology is \lq sh Lie-Rinehart pair\rq.
An \lq sh Lie-Rinehart pair\rq\ 
is there defined via suitable brackets.
In terms of the brackets \eqref{brannn} and \eqref{adjoint},
the $A$-multilinearity of $t$ is equivalent to 
\begin{align}
\{a x_1,\ldots,x_n|b\}&=(-1)^{|a|}a \{x_1,\ldots,x_n|b\},\ n \geq 1,
\label{sh1}
\end{align}
and \eqref{sh12} is equivalent to
\begin{equation}
\begin{aligned}
{}[x_1,\ldots,x_n, a x_{n+1}]&=\{x_1,\ldots,x_n|a\} x_{n+1} 
\\
&\quad
+(-1)^{(|x_1|+\ldots +|x_n|+n+1)|a|}
a
[x_1,\ldots,x_n, x_{n+1}],
\end{aligned}
\label{sh2}
\end{equation}
for $n \geq 1$, 
where $x_1,\ldots,x_n, x_{n+1}$ are homogeneous members of $L$ and $a,b$
homogeneous members of  $A$.
\end{rema}

\begin{rema}
\label{adjusted}
For $n=1$, the $A$-multilinearity of $t$ or,
equivalently, the condition
\eqref{sh1}, and 
the axiom
\eqref{sh12} or, equivalently, 
\eqref{sh2},
come down to \eqref{1.1.a}
and \eqref{1.1.b}, respectively,
adjusted to the graded situation.
\end{rema}

\begin{rema}
\label{common}
Write
the differential on $A$ as a unary bracket
$
\{\,\cdot\,\}\colon A \longrightarrow A
$
and that
on $L$ as a unary bracket
$
[\,\cdot\,]\colon L \longrightarrow L.
$
In terms of this notation, \eqref{sh2}
extends to the identity
\begin{align}
[ax]&=\{a\}x+(-1)^{|a|}a[x]
\label{sh2e}
\end{align}
involving  the unary brackets.
The identity \eqref{sh2e} precisely says that $L$ is a differential graded $A$-module.
A unitary bracket of the kind
$
[\,\cdot\,]\colon L \longrightarrow L
$
encapsulating a differential occurs, e.g.,
in \cite{MR2163405}.
\end{rema}

Let $(A,L,\cpartial,t)$
and  $(A',L',\cpartial',t')$
be two sh Lie-Rinehart algebras.
Given a morphism $\varphi \colon A \to A'$ of differential graded
algebras,
we view $(\Sigm_{A'}[sL'],d^0)$ as a differential graded
$A$-module via $\varphi$ in the obvious manner.

\begin{defi}\label{mor3}
A morphism
\begin{equation}
(\varphi,\Phi)\colon  (A,L,\cpartial,t)\longrightarrow (A',L',\cpartial',t')
\end{equation}
of sh actions
is a {\em morphism of sh Lie-Rinehart algebras\/}
when $\Phi$ passes to a morphism
$\Phi \colon (\Sigm_A[sL],d^0) \to (\Sigm_{A'}[sL'],d^0)$
of differential graded $A$-modules
(where the notation $\Phi$ is abused).
\end{defi}

This extends the notion of an ordinary morphism of Lie-Rinehart algebras, 
cf. \cite{poiscoho}.

\subsection{Multi derivation Maurer-Cartan characterization of an sh Lie-Rinehart structure}
\label{mcsh}

Let $(L,d^0)$ be a chain complex,
let 
$C=\Sigmc[sL]$ and write
the induced coalgebra differential  on $\Sigmc[sL]$ as $d^0$. 
Using the canonical isomorphism  $\Sigm[sL] \to  \Sigmc[sL]$
of differential graded cocommutative coalgebras,
cf. \eqref{cano1},
we interpret
\begin{equation}
\mathrm{Hom}(C,A)=  \mathrm{Hom}(\Sigmc[sL],A)
\end{equation} 
as the algebra
$\mathrm{Sym}(sL,A)$ of $A$-valued
$R$-multilinear graded symmetric functions on $sL$.

We apply the results in the previous section, with $V=sL$.
We maintain the notation \eqref{aexplicit}-\eqref{dhoa}.

\begin{thm}
\label{multishp}
Let
$(A,L,\cpartial,t)$ 
be an  sh Lie-Rinehart algebra.
For $j \geq 1$,
the derivation $\dho_j$ of $\mathrm{Hom}(\Sigm[sL],A)$
passes to a derivation of $\mathrm{Sym}_A(sL,A)=\mathrm{Hom}_A(\Sigm_A[sL],A)$,
and the resulting data of the kind
{\rm \eqref{multisym}}, viz.
\begin{equation}
(\mathrm{Sym}_A(sL,A),\dho_0, \dho_1, \dho_2, \ldots),
\label{multisym3}
\end{equation}
constitute a \mmd\  algebra.
\end{thm}

\begin{proof} Let $ j \geq 1$. Consider the operator 
$\dho_j= \apartial_j +\partial^{t_j}$ on 
$\mathrm{Hom}(\Sigm[sL],A)$.
By construction
\begin{align}
\dho_j&= \partial^{t_j}\colon A \to \mathrm{Hom}(\Sigm^j[sL],A)
\label{op1}
\\
\dho_j&= \apartial_j +\partial^{t_j}\colon
\mathrm{Hom}(sL,A) \to \mathrm{Hom}(\Sigm^{j+1}[sL],A).
\label{op2}
\end{align}
Let $u,a \in A$ and $\alpha_1,\ldots,\alpha_j \in  sL$
be homogeneous.
Since $t_j$ is $A$-multilinear,
\begin{align*}
\partial^{t_j}u(a\alpha_1,\ldots,\alpha_j)&=[t_j,u](a\alpha_1,\ldots,\alpha_j)\\
&= (-1)^{|u|(|a|+|\alpha_1|+ \ldots +|\alpha_j|)}t_j(a\alpha_1,\ldots,\alpha_j)(u)
\\
&=
(-1)^{|u||a|}
(-1)^{|u|(|\alpha_1|+ \ldots +|\alpha_j|)}
(-1)^{|a|}a t_j(\alpha_1,\ldots,\alpha_j)(u)
\\
&=(-1)^{(|u|+1)|a|}
a [t_j,u](\alpha_1,\ldots,\alpha_j)
\\
&=(-1)^{(|u|+1)|a|}
a(\partial^{t_j}u)(\alpha_1,\ldots,\alpha_j).
\end{align*}
Likewise, let
$a \in A$, $\varphi\in \mathrm{Hom}(sL,A)$ 
and $\alpha_1,\ldots,\alpha_j,\alpha_{j+1} \in  sL$
be homogeneous.
We claim first that
\begin{equation}
\begin{aligned}
\apartial_j\varphi(\alpha_1,\ldots,\alpha_j,a\alpha_{j+1})
&=(-1)^{(|\alpha_1|+\ldots +|\alpha_j|+a)|\varphi|+1}
t_j(\alpha_1,\ldots,\alpha_j)(a) \varphi(\alpha_{j+1})
\\
&\quad +(-1)^{(|\alpha_1|+\ldots +|\alpha_j|+1+|\varphi|)|a|}
a\apartial_j\varphi(\alpha_1,\ldots,\alpha_j, \alpha_{j+1}).
\end{aligned}
\label{redo1}
\end{equation}
Indeed,
by
\eqref{sh12}, viz.
\begin{align*}
\cpartial^j(\alpha_1,\ldots,\alpha_n, a \alpha_{j+1})
&= t_j(\alpha_1,\ldots,\alpha_j)(a) \alpha_{j+1}
\\
&\quad +(-1)^{(|\alpha_1|+\ldots +|\alpha_j|+1)|a|}
a\cpartial^j(\alpha_1,\ldots,\alpha_j, \alpha_{j+1})
\\
\apartial_j\varphi(\alpha_1,\ldots,\alpha_j,a\alpha_{j+1})
&=(-1)^{|\varphi|+1}
\varphi(\cpartial^j(\alpha_1,\ldots,\alpha_j,a\alpha_{j+1}))
\\
&=
(-1)^{|\varphi|+1} \varphi(t_j(\alpha_1,\ldots,\alpha_j)(a) \alpha_{j+1})
\\
&\quad
+ (-1)^{(|\alpha_1|+\ldots +|\alpha_j|+1)|a|+|\varphi|+1}
\varphi(a
\cpartial^j(\alpha_1,\ldots,\alpha_j, \alpha_{j+1}))
\\
&=(-1)^{(|\alpha_1|+\ldots +|\alpha_j|-1+a)|\varphi|+|\varphi|+1}
t_j(\alpha_1,\ldots,\alpha_j)(a) \varphi(\alpha_{j+1})
\\
&\quad +(-1)^{(|\alpha_1|+\ldots +|\alpha_j|+1+|\varphi|)|a|+|\varphi|+1}
a\varphi(
\cpartial^j(\alpha_1,\ldots,\alpha_j, \alpha_{j+1}))
\\
&=(-1)^{(|\alpha_1|+\ldots +|\alpha_j|+a)|\varphi|+1}
t_j(\alpha_1,\ldots,\alpha_j)(a) \varphi(\alpha_{j+1})
\\
&\quad +(-1)^{(|\alpha_1|+\ldots +|\alpha_j|+1+|\varphi|)|a|}
a\apartial_j\varphi(\alpha_1,\ldots,\alpha_j, \alpha_{j+1})
\end{align*}
whence \eqref{redo1}.

Next we claim
\begin{equation}
\begin{aligned}
\partial^{t_j}(\varphi)(\alpha_1,\ldots,\alpha_j,a\alpha_{j+1})
&=
(-1)^{(|\alpha_1|+\ldots+|\alpha_j|-1+|\varphi|)|a| }
a \partial^{t_j}(\varphi)(\alpha_1,\ldots,\alpha_j,\alpha_{j+1})
\\
&\quad +
 (-1)^{(|\alpha_1|+\ldots +|\alpha_j|+a)|\varphi|}
t_j(\alpha_1,\ldots,\alpha_j)(a) \varphi(\alpha_{j+1}).
\end{aligned}
\label{redo2}
\end{equation}
Indeed,
for $j=1$
\begin{align*}
[t_1,\varphi](\alpha_1,a\alpha_{2})
&=\bra\circ(t_1\otimes \varphi)(\alpha_1 \otimes (a\alpha_2)
+(-1)^{|\alpha_1|(|a|+|\alpha_2|)}(a\alpha_2) \otimes \alpha_1)
\\
&=
(-1)^{(|\alpha_1|-1+|\varphi|)|a| }
a [t_1,\varphi](\alpha_1,\alpha_{2})
+(-1)^{|\varphi|(|\alpha_1|+|a|)}t_1(\alpha_1)(a)\varphi(\alpha_2).
\end{align*}
The same kind of reasoning shows that,
for general $j \geq 1$,
\begin{align*}
[t_j,\varphi](\alpha_1,\ldots,\alpha_j,a\alpha_{j+1})
&=
(-1)^{(|\alpha_1|+\ldots+|\alpha_j|-1+|\varphi|)|a| }
a [t_j,\varphi](\alpha_1,\ldots,\alpha_j,\alpha_{j+1})
\\
&\quad +
 (-1)^{(|\alpha_1|+\ldots +|\alpha_j|+a)|\varphi|}
t_j(\alpha_1,\ldots,\alpha_j)(a) \varphi(\alpha_{j+1})
\end{align*}
whence \eqref{redo2}.

Combining \eqref{redo1} and \eqref{redo2},
since the summands involving
$t_j(\alpha_1,\ldots,\alpha_j)(a) \varphi(\alpha_{j+1})$
cancel out,
we find
\begin{align*}
\dho_j(\varphi)(\alpha_1,\ldots,\alpha_j,a\alpha_{j+1})
&= \apartial_j (\varphi)(\alpha_1,\ldots,\alpha_j,a\alpha_{j+1})
+\partial^{t_j}(\varphi)(\alpha_1,\ldots,\alpha_j,a\alpha_{j+1})
\\
&=(-1)^{(|\alpha_1|+\ldots+|\alpha_j|+|\varphi|+1)|a|}
a\apartial_j\varphi(\alpha_1,\ldots,\alpha_j,\alpha_{j+1})
\\
&\quad
+(-1)^{(|\alpha_1|+\ldots+|\alpha_j|+|\varphi|+1)|a|}
a[t_j,\varphi](\alpha_1,\ldots,\alpha_j,\alpha_{j+1})
\\
&=(-1)^{(|\alpha_1|+\ldots+|\alpha_j|+|\varphi|+1)|a|}
a\dho_j(\varphi)(\alpha_1,\ldots,\alpha_j,\alpha_{j+1})
\end{align*}
as asserted. \end{proof}

Under the circumstances of Theorem \ref{multishp},
we refer to the \multialgebra\  \eqref{multisym3}
as the {\em \mmd\  algebra
associated to the sh-Lie-Rinehart algebra\/} $(A,L,\cpartial,t)$.

\begin{rema}
\label{classical3}
By construction, relative to the
filtration degree and the complementary degree,
the \mmd\  algebra
associated to an sh-Lie-Rinehart algebra
is actually a bigraded  \mmd\  algebra, cf. 
Remarks \ref{classical1} and \ref{classical} above.
\end{rema}

The \mmd\  algebra associated to an sh Lie-Rinehart algebra
is natural in a sense we now explain.

\begin{defi}\label{mor4}
A comorphism
\[
(\varphi,\Phi)\colon (A',L',\cpartial',t') \longrightarrow (A,L,\cpartial,t)
\]
of sh actions is a {\em comorphism
of Lie-Rinehart algebras\/}  
when $\Phi\colon \Sigm[sL'] \to \Sigm[sL]$
passes to a morphism
$\widehat \Phi\colon \Sigm_{A'}[sL'] \to A'\otimes_A\Sigm_A[sL]$
of differential graded $A'$-modules 
making the diagram
\begin{equation}
\begin{CD}
\Phi\colon \Sigm_{A'}[sL'] @>{\widehat \Phi}>>  A'\otimes_A\Sigm_A[sL]
\\
@V{t'}VV
@V{t^{\flat}}VV
\\
 \mathrm{Der}(A'|R) @>{\varphi^{\flat}}>>\mathrm{Der}(A, A') 
\end{CD}
\label{shder2}
\end{equation}
commutative. 
\end{defi}
Proposition \ref{como1} now extends to the following.

\begin{prop}
\label{como2}
A comorphism
$(\varphi,\Phi)\colon (A',L',\cpartial',t') \longrightarrow (A,L,\cpartial,t)$
of sh Lie-Rine\-hart algebras induces a morphism
\begin{equation}
(\varphi_*,\Phi^*)\colon 
(\mathrm{Sym}_A(sL,A), \dho_0,\dho_1,\dho_2,\ldots)
\longrightarrow (\mathrm{Sym}_{A'}(sL',A'),\dho_0,\dho'_1,\dho'_2+\ldots )
\label{como23}
\end{equation}
between the associated \mmd\  algebras.
\end{prop}

We note that, by construction,
\[
\mathrm{Sym}_A(sL,A) \longrightarrow \mathrm{Sym}_{A'}(sL',A')
\]
is the composite of
\[
\varphi_*\colon \mathrm{Hom}_A(\Sigm_A[sL],A) \longrightarrow 
\mathrm{Hom}_{A'}(A'\otimes_A\Sigm_A[sL],A')
\]
with
\[
\Phi^*\colon \mathrm{Hom}_{A'}(A'\otimes_A\Sigm_A[sL],A')
\longrightarrow \mathrm{Hom}_{A'}(\Sigm_{A'}[sL'],A').
\]
Proposition \ref{como2} says that this composite
is compatible with the \mmd\  algebra structures.

The comorphism concept for sh Lie-Rinehart algebras
generalizes that for ordinary Lie-Rinehart 
algebras, cf. \cite{MR1235995}, where this is explained for Lie algebroids.
As already noted in Remark \ref{variance}, for reasons of variance, 
a morphism 
\[
(\varphi,\Phi)\colon  (A,L,\cpartial,t)\longrightarrow (A',L',\cpartial',t')
\]
of sh Lie-Rinehart algebras,
as defined in Subsection \ref{shlr}  above, does not induce a
morphism between the associated \mmd\  algebras in 
an obvious manner except when $A=A'$ and $\varphi$
is the identity---in this case the notions of morphism and comorphism 
coincide---,
not even for the special case of 
a morphism of
ordinary Lie-Rinehart algebras (except when $A=A'$ and $\varphi$ 
is the identity).

\begin{thm}
\label{shlrchar0}
Let $A$ be a differential graded commutative algebra and
$L$ a differential graded
$A$-module having the property that the canonical
$A$-module morphism from $L$ to its double $A$-dual
is injective.
Let $\cpartial$ be a \ppt\ 
on $\Sigm[sL]$
and $t$ a \fdm\  of the kind 
{\rm \eqref{tc}},
and let
$\dho_0$ denote the algebra differential
on $\mathrm{Sym}_A(sL,A)$ induced from the differentials
on $L$ and $A$. Then
$(A,L,\cpartial,t)$ 
is an  sh Lie-Rinehart algebra
if and only if
$(\mathrm{Sym}_A(sL,A),\dho_0, \dho_1, \dho_2, \ldots)$
is a \multialgebra,
necessarily the \mmd\  algebra 
associated to $(A,L,\cpartial,t)$.
\end{thm}

\begin{proof} Theorem \ref{multishp} says that the condition is necessary.
Thus suppose that 
\[
(\mathrm{Sym}_A(sL,A),\dho_0, \dho_1, \dho_2, \ldots)
\]
is a \multialgebra.
Theorem \ref{char2} implies that $(\Sigm[sL],\cpartial)$ is a differential graded cocommutative coalgebra, that is,
that $(L,\cpartial)$ is an sh Lie algebra, and that
$t\colon \Sigm[sL] \to \mathrm{Der}(A|R)$
is a Lie algebra twisting cochain. 
Thus it remains to show that $t$ is $A$-multilinear and satisfies
the axiom \eqref{sh12}. Formally exactly the same reasoning as that 
in the proof of Corollary \ref{char1} above establishes these claims.

Reading backwards the reasoning in the proof of Theorem \ref{multishp}
yields the details: 
Indeed, let $j \geq 1$.
Let $u,a \in A$ and $\alpha_1,\ldots,\alpha_j \in  sL$
be homogeneous.
Since $\apartial_j(u)=0$ and since
$\partial^{t_j}u=[t_j,u]$, the hypothesis  implies that
\begin{align*}
t_j (a\alpha_1,\ldots,\alpha_j)(u)
&=(-1)^{(|a|+|\alpha_1| +\ldots +|\alpha_j|)|u|}
[t_j, u](a\alpha_1,\ldots,\alpha_j)
\\
&=(-1)^{(|a|+|\alpha_1| +\ldots +|\alpha_j|)|u|}(-1)^{|a|(|u|-1)}a[t_j, u](\alpha_1,\ldots,\alpha_j)
\\
&=
(-1)^{|a|}at_j (\alpha_1,\ldots,\alpha_j)(u)
\end{align*}
whence, since $u$ is arbitrary, 
$t_j$ is $A$-multilinear.

Likewise, let
$a \in A$, $\varphi\in \mathrm{Hom}(sL,A)$ 
and $\alpha_1,\ldots,\alpha_j,\alpha_{j+1} \in  sL$
be homogeneous.
We already know that \eqref{redo2} holds, viz.
\begin{equation*}
\begin{aligned}
\partial^{t_j}(\varphi)(\alpha_1,\ldots,\alpha_j,a\alpha_{j+1})
&=
(-1)^{(|\alpha_1|+\ldots+|\alpha_j|-1+|\varphi|)|a| }
a \partial^{t_j}(\varphi)(\alpha_1,\ldots,\alpha_j,\alpha_{j+1})
\\
&\quad +
 (-1)^{(|\alpha_1|+\ldots +|\alpha_j|+a)|\varphi|}
t_j(\alpha_1,\ldots,\alpha_j)(a) \varphi(\alpha_{j+1}).
\end{aligned}
\end{equation*}
Since the operator $\dho_j$ passes to $\mathrm{Sym}_A(sL,A)$,
we conclude that \eqref{redo1} holds, viz.
\begin{equation*}
\begin{aligned}
\apartial_j\varphi(\alpha_1,\ldots,\alpha_j,a\alpha_{j+1})
&=(-1)^{(|\alpha_1|+\ldots +|\alpha_j|+a)|\varphi|+1}
t_j(\alpha_1,\ldots,\alpha_j)(a) \varphi(\alpha_{j+1})
\\
&\quad +(-1)^{(|\alpha_1|+\ldots +|\alpha_j|+1+|\varphi|)|a|}
a\apartial_j\varphi(\alpha_1,\ldots,\alpha_j, \alpha_{j+1}).
\end{aligned}
\end{equation*}
Moreover, 
\begin{equation}
\apartial \varphi(\alpha_1, \ldots, \alpha_j, a\alpha_{j+1}) =
(-1)^{|\varphi|+1}\varphi(\cpartial
(\alpha_1,  \ldots, \alpha_j, a\alpha_{j+1})).
\end{equation}
Consequently 
\begin{equation}
\begin{aligned}
\varphi(\cpartial^j(\alpha_1,\ldots,\alpha_j, a \alpha_{j+1}))
&=
\varphi( t_j(\alpha_1,\ldots,\alpha_j)(a) \alpha_{j+1})
\\
&\quad
+(-1)^{(|\alpha_1|+\ldots +|\alpha_j|+1)|a|}
\varphi(a\cpartial^j(\alpha_1,\ldots,\alpha_j, \alpha_{j+1})) .
\label{sh121}
\end{aligned}
\end{equation}
Since $\varphi$ is arbitrary, the hypothesis implies that
the identity \eqref{sh12} holds.
\end{proof}

Theorems \ref{char2} and \ref{olr2} imply the following:
\begin{thm}
\label{shlrchar}
Let $A$ be a differential graded commutative algebra and
$L$ a differential graded
$A$-module having the property that the canonical
$A$-module morphism from $L$ to its double $A$-dual
is an isomorphism, and let
$\dho_0$ denote the algebra differential
on $\mathrm{Sym}_A(sL,A)$ induced from the differentials
on $L$ and $A$.
Sh Lie-Rinehart structures on $(A,L)$
extending the differentials on $A$ and $L$
and \multialgebra\  structures on
$\mathrm{Sym}_A(sL,A)$ extending the algebra differential 
$\dho_0$ are equivalent notions.
The equivalence between the two notions is that spelled out explicitly in
Theorem  {\rm \ref{shlrchar0}}.
\end{thm}

\begin{rema}
\label{remprecise}
Theorem \ref{shlrchar0} says that,
given the data $(A,L,\cpartial,t)$,
under the hypothesis spelled out there,
these data constitute an sh Lie-Rinehart algebra if and only if
they induce a \multialgebra\  structure on
$\mathrm{Sym}_A(sL,A)$.
On the other hand,
Theorem \ref{shlrchar} says that, 
under the stronger hypothesis of this theorem,
every \multialgebra\  structure on
$\mathrm{Sym}_A(sL,A)$
of the kind under discussion arises from a unique
sh Lie-Rinehart algebra structure on $(A,L)$.
\end{rema}

\subsection{Quasi Lie-Rinehart algebras}
\label{quasilr}
Let $\AA$ be a differential 
graded commutative algebra concentrated
in degrees $\leq 0$ and,
as before, we then write
$\AA^j=\AA_{-j}$ so that
$\AA^j=0$ for $j <0$.
Furthermore, let $\QQ$ be a differential
graded $\AA$-module whose underlying graded
$\AA$-module is an induced module of the kind
$\QQ=\AA\otimes _A Q$ where $A=\AA^0$ and where
$Q$ is concentrated in degree zero.
Suppose that the canonical $A$-module morphism from
$Q$ to its double $A$-dual is an isomorphism.

In \cite{MR2103009} we introduced quasi Lie-Rinehart algebras.
We recall that
a quasi-Lie-Rinehart algebra structure on $(\AA,\QQ)$ involves 
the following three items:

\noindent
---  a graded skew-symmetric
$R$-bilinear pairing of degree zero
\begin{equation}
\bra_{\QQ}
\colon
\QQ
\otimes_R
\QQ
\xrightarrow{\ }
\QQ,
\label{braQ}
\end{equation}

\noindent
---  an $R$-bilinear pairing of degree zero
\begin{equation}
\act
\colon \QQ \otimes_R \AA
\xrightarrow{\ }
\AA,
\quad (\xi,\alpha) \mapsto \xi(\alpha),\ \xi \in \QQ,\ \alpha \in \AA,
\label{paiQ}
\end{equation}
such that, given $\xi \in \QQ$, the operation  $(\xi,\,\cdot\,)$
is a homogeneous $R$-linear
derivation of $\AA$ of degree $|\xi|$;

\noindent
---  an $A$-trilinear operation of degree $-1$
(beware: in upper degrees)
\begin{equation}
\actt_{\QQ} 
\colon
Q
\otimes_{A}
Q
\otimes_{A}
\AA
\xrightarrow{\ }
\AA
\label{triQ}
\end{equation}
which is graded skew-symmetric in the first two variables
(i.~e. in the $Q$-variables), such that,
given $\xi, \theta \in Q$, the operation 
$\langle \xi,\theta;\,\cdot\,\rangle$
is a homogeneous $A$-linear derivation of $\AA$ of degree $-1$.
These pieces of structure are subject to a number of constraints, 
see \cite{MR2103009} for details. We do not spell out these constraints here;
instead we will now explain directly the Maurer-Cartan algebra characterization
of the structure. This will illustrate the technology developed in
the present paper. The sign of the Lie-Rinehart operator
(generalized CCE operator)
in 
\cite{MR2103009}
is the 
negative of the sign of the Lie-Rinehart operator
in the present paper.

We note first that
\begin{equation}
\mathrm{Sym}_{\AA}(s\QQ,\AA)=
\mathrm{Hom}_{\AA}(\Sigm[s\QQ],\AA) \cong \Aalt_A(Q,\AA).
\label{identifi1}
\end{equation}
Now, let $\cpartial^1\colon \Sigm[s\QQ] \to \Sigm[s\QQ]$
denote the coderivation induced by \eqref{braQ}
and, likewise, 
$t_1\colon \Sigm^1[s\QQ] = s\QQ \to \mathrm{Der}(\AA|R)$
the degree $-1$ morphism of $R$-modules
given by the composite of the desuspension with the adjoint of
\eqref{paiQ}. As before, let
$\apartial_1$ and $\partial^{t_1}$
denote the derivations on $\mathrm{Hom}(\Sigm[s\QQ],\AA)$
induced by 
$\cpartial^1$ and $t_1$, respectively, cf. \eqref{aexplicit} and
\eqref{expres1}.
The compatibility conditions among the pieces of structure
\eqref{braQ} and \eqref{paiQ}
(not spelled out here) imply that the sum
$\dho_1= \apartial_1 +\partial^{t_1}$
passes to a derivation $\dho_1$ on $\Aalt_A(Q,\AA)$,
formally a CCE operator with respect to \eqref{braQ} and \eqref{paiQ}.
Indeed, given $\xi_1,\dots,\xi_{p+1} \in Q$.
\begin{equation*}
\begin{aligned}
{}&(-1)^{|f|}(\dho_1f)(\xi_1,\dots,\xi_{p+1})
= 
\sum_{j=1}^{p+1} 
(-1)^{j-1}
\xi_j(f(\xi_1,\dots\widehat{\xi_j}\dots,\xi_{p+1}))
\\
&+
\sum_{1 \leq j<k \leq p+1} (-1)^{j+k} f(\lbrack \xi_j,\xi_k \rbrack_Q,
\xi_1,\dots\widehat{\xi_j}\dots\widehat{\xi_k}\dots,\xi_{p+1});
\end{aligned}
\end{equation*}
here the notation $\bra_Q$
refers to the restriction to the degree zero constituent
$Q$ (of $\QQ$) of the bracket \eqref{braQ} on $\QQ$.
Next,
given
$f \in \Aalt^p_A(Q,\AA^q)=\mathrm{Hom}_A(\Lambda^p_A[sQ],\AA^q)$,
define the value $\dho_2(f)(\xi_1,\dots,\xi_{p+2})$ by
\begin{equation}
\begin{aligned}
{}&(-1)^{|f|}(\dho_2f)(\xi_1,\dots,\xi_{p+2})=
\\
{}&
\sum_{1 \leq j<k \leq p+2}  
(-1)^{j+k}
\langle\xi_j,\xi_k;
f(\xi_1, \dots\widehat{\xi_j}\dots\widehat{\xi_k}\dots,\xi_{p+2})\rangle_Q 
\end{aligned}
\label{4.8.2}
\end{equation}
where $\xi_1,\dots,\xi_{p+2} \in Q$.

\begin{rema} The expression in \cite{MR2103009} (4.8.2)
that corresponds to \eqref{4.8.2} above
involves the additional sign factor $(-1)^p$.
This sign factor has its origin in the relationship with Lie-Rinehart triples,
cf. \cite{MR2103009} (4.11) and (4.12).
More precisely, in terms of the notation in  \cite{MR2103009},
given
$x_1,\dots,x_{q-1} \in H$, 
$\alpha \in \Aalt^q(H,A)=\AA^q$,
in \cite{MR2103009} (4.11.7),
the value of the pairing \eqref{triQ} is defined by
\begin{equation*}
\langle\xi,\eta;\alpha\rangle_Q(x_1,\dots,x_{q-1})
=
\alpha(\delta(\xi,\eta),x_1,\dots,x_{q-1});
\end{equation*}
however, when we define that value by
\begin{equation}
\langle\xi,\eta;\alpha\rangle_Q(x_1,\dots,x_{q-1})
=(-1)^{|\alpha|}
\alpha(\delta(\xi,\eta),x_1,\dots,x_{q-1}),
\label{4.11.7mo}
\end{equation}
the factor $(-1)^p$ in \cite{MR2103009} (4.8.2) disappears.
\end{rema}

Let $\dho_0$ denote the algebra differential
on $\Aalt_A(Q,\AA)$
induced from the differentials on $\AA$ and $\QQ$.
By \cite{MR2103009} (Theorem 4.10),
\eqref{braQ}, \eqref{paiQ},
and
\eqref{triQ} turn
$(\AA,\QQ)$
into a quasi Lie-Rinehart
algebra if and only if
$\left(\Aalt_A(Q,\AA),\dho_0,\dho_1,\dho_2\right)$
is a multi algebra, and any
quasi Lie-Rinehart algebra structure arises in this manner.
In view of \eqref{identifi1},
Theorem \ref{shlrchar} shows that
the multi algebra 
$\left(\Aalt_A(Q,\AA),\dho_0,\dho_1,\dho_2\right)$, in turn,
characterizes a special kind of
sh Lie-Rinehart structure on $(\AA,\QQ)$.
Hence a quasi Lie-Rinehart
structure on $(\AA,\QQ)$ is equivalent to a special kind of
sh Lie-Rinehart structure.

Let $\left(\AA,\QQ,\bra_{\QQ},\act,\actt\right)$
be a quasi Lie-Rinehart algebra.
With hindsight, it is now interesting to
spell out the corresponding
sh Lie-Rinehart structure  on $(\AA,\QQ)$
arising from  the quasi Lie-Rinehart
algebra structure.
This is straightforward: The bracket
\eqref{braQ} is already defined and, as before,
we define
$t_1\colon \Sigm^1[s\QQ] = s\QQ \to \mathrm{Der}(\AA|R)$
to be the degree $-1$ morphism of $R$-modules
given by the composite of the desuspension with the adjoint of
\eqref{paiQ}. 
Now, noting that a general sh Lie algebra structure
on $\QQ$ is given by a system of brackets of the kind
\begin{equation*}
\bran_n\colon  \QQ^{j_1} \times \ldots  \times \QQ^{j_n}
\longrightarrow  \QQ^{j_1+\ldots +j_n-n+2}, \ n \geq 2,\ 
j_1,\ldots,j_n \geq 0,
\end{equation*} 
we see that the 3-variable bracket we are looking for must be of the kind
\[
\braaa\colon \QQ^{j_1} \times  \QQ^{j_2}  \times \QQ^{j_3}
\longrightarrow  \QQ^{j_1+j_2 +j_3-1},\ j_1,j_2,j_3 \geq 0;
\]
in particular, the constituent of $\braaa$
on the degree zero component $Q \times Q \times Q$
of $\QQ\times \QQ\times \QQ$
is necessarily zero.
Define the operation
\begin{equation}
\braaa
\colon
Q
\otimes
Q
\otimes
\QQ
\xrightarrow{\ }
\QQ
\label{triQo}
\end{equation}
by setting
\begin{equation}
[x_1,x_2,ax_3]_3=\langle x_1,x_2;a\rangle_{\QQ} x_3,\ x_1,x_2,x_3 \in Q,\ a \in \AA.
\label{triQo0}
\end{equation}
Write the operations
$\act$ and $\actt_{\QQ}$ as
\begin{align*}
\{\,\cdot\,|\,\cdot\,\}
&\colon
\QQ
\otimes
\AA
\xrightarrow{\ }
\AA
\\
\{\,\cdot\, , \,\cdot\,|\,\cdot\,\}
&\colon
Q
\otimes_{A}
Q
\otimes_{A}
\AA
\xrightarrow{\ }
\AA,
\end{align*}
respectively
and, using \eqref{sh1}, extend 
$\{\,\cdot\, , \,\cdot\,|\,\cdot\,\}$ to an operation
\begin{equation}
\{\,\cdot\, , \,\cdot\,|\,\cdot\,\}
\colon
\QQ
\otimes_{A}
\QQ
\otimes_{A}
\AA
\xrightarrow{\ }
\AA .
\end{equation}
Thereafter, 
extend \eqref{triQo} to 
\begin{equation}
\braaa
\colon
\QQ
\otimes_{R}
\QQ
\otimes_{R}
\QQ
\xrightarrow{\ }
\QQ
\label{triQO}
\end{equation}
in such a way that
\eqref{sh1} holds, i.~e., given homogeneous $x_1,x_2,x_3 \in \QQ$
and $a \in \AA$,
\begin{equation}
[x_1,x_2,ax_3]_3=\{x_1,x_2|a\} x_3 +(-1)^{(|x_1|+|x_2|+1)|a|}a[x_1,x_2,x_3]_3 .
\end{equation}
Then $\left(\AA,\QQ,\bra_{\QQ},\braaa,\{\,\cdot\,|\,\cdot\,\},\{\,\cdot\, , \,\cdot\,|\,\cdot\,\}\right)$
is the sh Lie-Rinehart algebra we are looking for.

It is also interesting to spell out explicitly
how the 3-variable bracket controls the failure of the
2-variable bracket to satisfy the Jacobi identity:
By construction, the control of the failure of the 2-variable bracket
to satisfy the Jacobi identity is encapsulated in the identity
\eqref{benefit2}, viz.
\begin{equation}
d^0 \cpartial^2 + \cpartial^1 \cpartial^1 + \cpartial^2 d^0=0.
\label{ident8}
\end{equation}
Given $\xi,\eta,\vartheta \in Q$,
for degree reasons,
$[\xi,\eta,\vartheta]$ is zero, 
as noted already,
and \eqref{ident8}
entails
\begin{align*}
\sum_{(\xi,\eta,\vartheta)\ \mathrm{cyclic}}
[[\xi,\eta]_Q,\vartheta]_Q
=
&
\pm \left([d_0\xi,\eta,\vartheta]_3
+[\xi,d_0\eta,\vartheta]_3
+[\xi,\eta,d_0\vartheta]_3
\right).
\end{align*}
In the Lie-Rinehart triple case
(a concept not explained here), this identity is equivalent to
 \cite{MR2103009} (1.9.6).
See also  \cite{MR2103009} (2.8.5(v)) and
the proof of Theorem 4.10 in \cite{MR2103009}, expecially item (v)
in this proof.

In \cite{MR2103009}  we have shown that
a foliation determines a Lie-Rinehart triple
and that a Lie-Rinehart triple determines a
quasi Lie-Rinehart algebra.
This quasi Lie-Rinehart algebra encapsulates
the higher homotopies behind the foliation.
In particular, it has the spectral sequence of the foliation
as an invariant of the structure.

\addcontentsline{toc}{section}{References}
\def\cprime{$'$} \def\cprime{$'$} \def\dbar{\leavevmode\hbox to 0pt{\hskip.2ex
  \accent"16\hss}d} \def\cprime{$'$} \def\cprime{$'$} \def\cprime{$'$}
  \def\cprime{$'$} \def\cprime{$'$}
  \def\polhk#1{\setbox0=\hbox{#1}{\ooalign{\hidewidth
  \lower1.5ex\hbox{`}\hidewidth\crcr\unhbox0}}}

\end{document}